\documentclass[10pt, a4paper]{article}
\usepackage[margin=1in]{geometry}
\usepackage{float}

\usepackage{graphicx}
\usepackage{xcolor, color}
\usepackage{amsthm, amsfonts,amssymb} 
\usepackage{lmodern}
\usepackage[T1]{fontenc}
\usepackage{mathtools}

\newtheorem{thm}{Theorem}[section]
\newtheorem{lem}[thm]{Lemma}

\newtheorem{mydef}{Definition}[section]

\newtheorem{rem}{Remark}[section]
\theoremstyle{remark}
\newcommand\bb[1]{\mathbf{#1}}

\newcommand{\R}{\mathbb{R}}
\newcommand{\ep}{\varepsilon}

\newcommand{\weak}{\rightharpoonup}
\newcommand{\di}{\, {\rm d}}
\newcommand{\dd}{{\rm d}}
\date{}
\newcommand\bint[1]{\displaystyle\int #1}

\usepackage[font=footnotesize,labelfont=bf]{caption}

\title{On time-periodic solutions to an interaction problem between compressible viscous fluids and viscoelastic beams}

\author{Ond\v rej Kreml, V\' aclav M\' acha, \v S\' arka Ne\v casov\' a, Sr\dj an Trifunovi\' c}

\newcommand{\Addresses}{{
  \bigskip
  \footnotesize

Ond\v rej Kreml, \textsc{Institute of Mathematics of the Academy of Sciences of the Czech Republic}\par\nopagebreak
  \textit{E-mail address}: \texttt{kreml@math.cas.cz}

\medskip

V\'{a}clav M\'{a}cha, \textsc{Institute of Mathematics of the Academy of Sciences of the Czech Republic}\par\nopagebreak
  \textit{E-mail address}: \texttt{macha@math.cas.cz}

\medskip

\v{S}\'{a}rka Ne\v{c}asov\'{a}, \textsc{Institute of Mathematics of the Academy of Sciences of the Czech Republic}\par\nopagebreak
  \textit{E-mail address}: \texttt{matus@math.cas.cz}
  
\medskip

Sr\dj{}an Trifunovi\'{c}, \textsc{Department of Mathematics and Informatics, Faculty of Sciences, University of Novi Sad}\par\nopagebreak
  \textit{E-mail address}: \texttt{srdjan.trifunovic@dmi.uns.ac.rs}

}}

\begin{document}
\maketitle
{\begin{abstract}
    In this paper, we study a nonlinear fluid-structure interaction problem between a viscoelastic beam and a compressible viscous fluid. The beam is immersed in the fluid which fills a two-dimensional rectangular domain with periodic boundary conditions. Under the effect of periodic forces acting on the beam and the fluid, at least one time-periodic weak solution is constructed which has a bounded energy and a fixed prescribed mass.
\end{abstract}}

\\${}$ 

\section{The model}
Let $L,H,T>0$ and define \begin{equation*}
    \Gamma:= (0,L), \quad \Omega = (0,L)\times (-H,H).
\end{equation*}
We denote the horizontal variable by $x$ and the vertical variable by $z$. The fluid fills the domain $\Omega$ and it is described with velocity $\bb{u}:(0,T)\times\Omega\to \mathbb{R}^2$ and density $\rho:(0,T)\times\Omega\to \mathbb{R}$ which are periodic in both the $x$ and the $z$ direction. The beam is immersed in the fluid and its vertical displacement is given as $\eta:(0,T)\times\Gamma\to \mathbb{R}$, while its graph is denoted as 
\begin{equation*}
    \Gamma^\eta(t) := \{ (x,\eta(t,x)) : x\in \Gamma \}.
\end{equation*}
In order to work on a fixed domain $\Omega$ (note that $\eta$ does not necessarily have values in $[-H,H]$), let us define a $z$-periodic version of $\eta$
\begin{equation*}
    \hat\eta(t,x):=\eta(t,x)-2n(t,x)H, 
\end{equation*}
where $n(t,x)\in \mathbb Z$ is uniquely determined by the requirement $\eta(t,x)-2n(t,x)H\in [-H,H).$
Its graph $\hat{\Gamma}^\eta(t)$ is on Figure \ref{figure1}.
\begin{figure}[ht]
    \centering\makebox[\textwidth]{\includegraphics[width=\columnwidth]{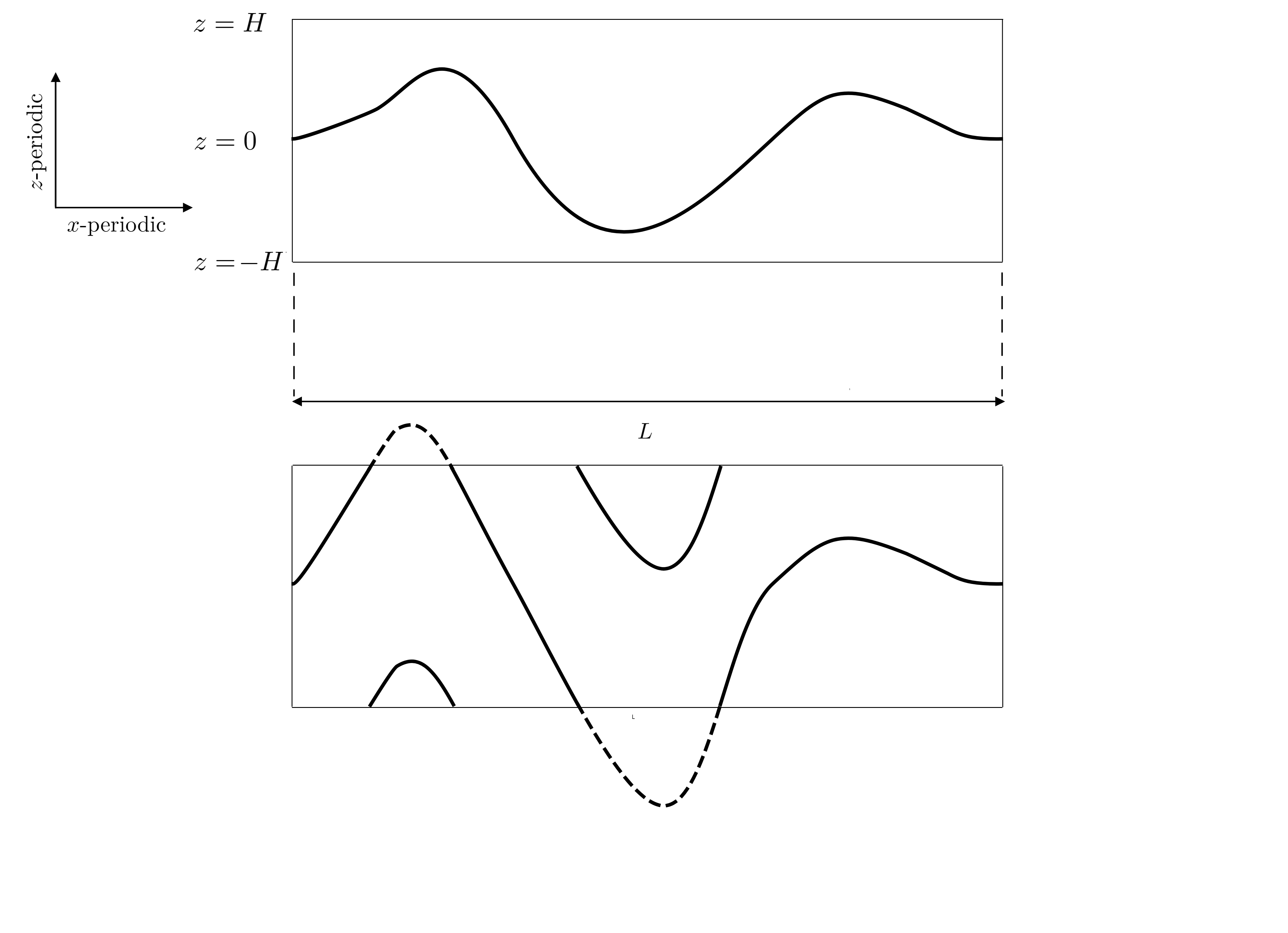}}
    \caption{Two examples of the beam inside the fluid. On the top, the structure is completely contained in $\Omega$ so $\Gamma^\eta(t)=\hat{\Gamma}^\eta(t)$. On the bottom, the structure leaves $\Omega$ and re-enters from the other side so $\Gamma^\eta(t)\neq\hat{\Gamma}^\eta(t)$ (the dashed part represents $\Gamma^\eta(t)\setminus\hat{\Gamma}^\eta(t)$).}
    \label{figure1}
\end{figure}
The time-space cylinders corresponding to our problem will be denoted as
\begin{equation*}
		Q_T:=(0,T)\times\Omega, \quad \Gamma_T:=(0,T)\times\Gamma.
\end{equation*}

\noindent
The governing equations for our coupled fluid-structure interaction problem read as follows:\\

\noindent
\textbf{The viscoelastic beam equation} on $\Gamma_T$:
\begin{equation}\label{structureeqs}
\eta_{tt}+ \eta_{xxxx} -  \eta_{txx}=-S^\eta \bb{f}_{fl}\cdot \mathbf{e_2} + f. 
\end{equation}
Here $f$ denotes a given external time-periodic force acting on the viscoelastic beam and $\bb{f}_{fl}$ is the force with which the fluid acts on the beam. Moreover, $S^\eta=\sqrt{1+|\eta_x|^2}$ is the Jacobian of the transformation from Eulerian to Lagrangian coordinates of the beam (i.e. from $\Gamma^\eta$ to $\Gamma$). \\

\noindent
\textbf{The compressible Navier-Stokes equations} on $\bigcup_{t\in(0,T)}\{t\}\times  \big(\Omega\setminus\hat{\Gamma}^{\eta}(t)\big)$:
\begin{equation}
\begin{split}
    \partial_t (\rho\mathbf{u}) + \nabla \cdot (\rho\mathbf{u}\otimes \mathbf{u})& = -\nabla p(\rho) +\nabla \cdot \mathbb{S}(\nabla \bb{u})+\rho \bb{F}, \\ 
    \partial_t \rho + \nabla \cdot (\rho \mathbf{u}) &= 0, 
\end{split}
\end{equation}
where we set the pressure $p$ for simplicity to be
\begin{equation*}
p(\rho)=\rho^\gamma,    
\end{equation*}
the viscous stress tensor $\mathbb{S}$ is given by the Newton rheological law
\begin{equation*}
    \mathbb{S}(\nabla \bb{u}):=\mu \big( \nabla \bb{u} + \nabla^\tau \bb{u}- \nabla \cdot \bb{u} \mathbb{I} \big) + \zeta \nabla \cdot \bb{u}\mathbb{I},\quad \mu,\zeta >0,
\end{equation*}
and $\bb{F}$ is a given time-periodic force acting onto the fluid.\\

\noindent
\textbf{The fluid-structure coupling (kinematic and dynamic, resp.)} on $\Gamma_T$:
\begin{eqnarray}
    \eta_t (t,x) \mathbf{e_2}&=& \mathbf{u}(t,x,\hat{\eta}(t,x)), \label{kinc}\\
    \bb{f}_{fl}(t,x)&=&\big[\big[(-p(\rho)\mathbb{I}+\mathbb{S}(\nabla \bb{u}))\big]\big](t,x,\hat\eta(t,x))~\nu^\eta(t,x), \label{dync}
\end{eqnarray}
where $\nu^\eta=\frac{(-\eta_x,1)}{\sqrt{1+|\eta_x|^2}}$ denotes the normal vector on $\Gamma^\eta$ facing upwards and
\begin{equation*}
    [[A]](\cdot,z):= \lim\limits_{\varepsilon\to 0^+} \big(A(\cdot,z-\varepsilon)-A(\cdot,z+\varepsilon)\big)
\end{equation*}
represents the jump of quantity $A$ in the vertical direction.\\

\noindent
\textbf{The beam boundary conditions}:
\begin{equation}
    \eta \text{ is periodic in } x \text{ and } \eta(t, x)= 0,
    \quad (t,x)\in (0,T)\times \{0,L\}.\label{boundaryconditions1}
\end{equation}

\noindent
\textbf{Fluid spatial periodicity}: 
\begin{equation}
    \rho,\bb{u} \text{ are periodic in }x \text{ and } z\text{ directions}.
\end{equation}

\noindent
\textbf{Time periodicity}: 
\begin{equation}
    \rho,\bb{u}, \eta \text{ are periodic in time}.\label{eq:timeper}
\end{equation}

\section{Weak solution and main result}
The nature of the studied problem enables us to work with two equivalent formulations of the problem. In the original formulation, the domain $\Omega$ is fixed and the viscoelastic beam appears inside the domain $\Omega$. However, we may use the $z-$periodicity of the problem to formulate it on the moving domain $\Omega^\eta(t)$ filled with the fluid, where the top and the bottom of the domain is given by the viscoelastic beam. For a given $\eta(t,x)$ we introduce an equivalent fluid domain and the corresponding time-space cylinder
\begin{equation}\label{eqv:domain}
    \Omega^\eta(t):=\{(x,z): x\in(0,L), \eta(t,x)<z<\eta(t,x)+2H \}, \qquad Q_T^\eta:=\bigcup_{t\in (0,T)} \{t\}\times \Omega^\eta(t),
\end{equation}
both domains are demostrated in Figure \ref{figure2}.

\begin{figure}[H]
    \centering\makebox[\textwidth]{\includegraphics[width=\columnwidth]{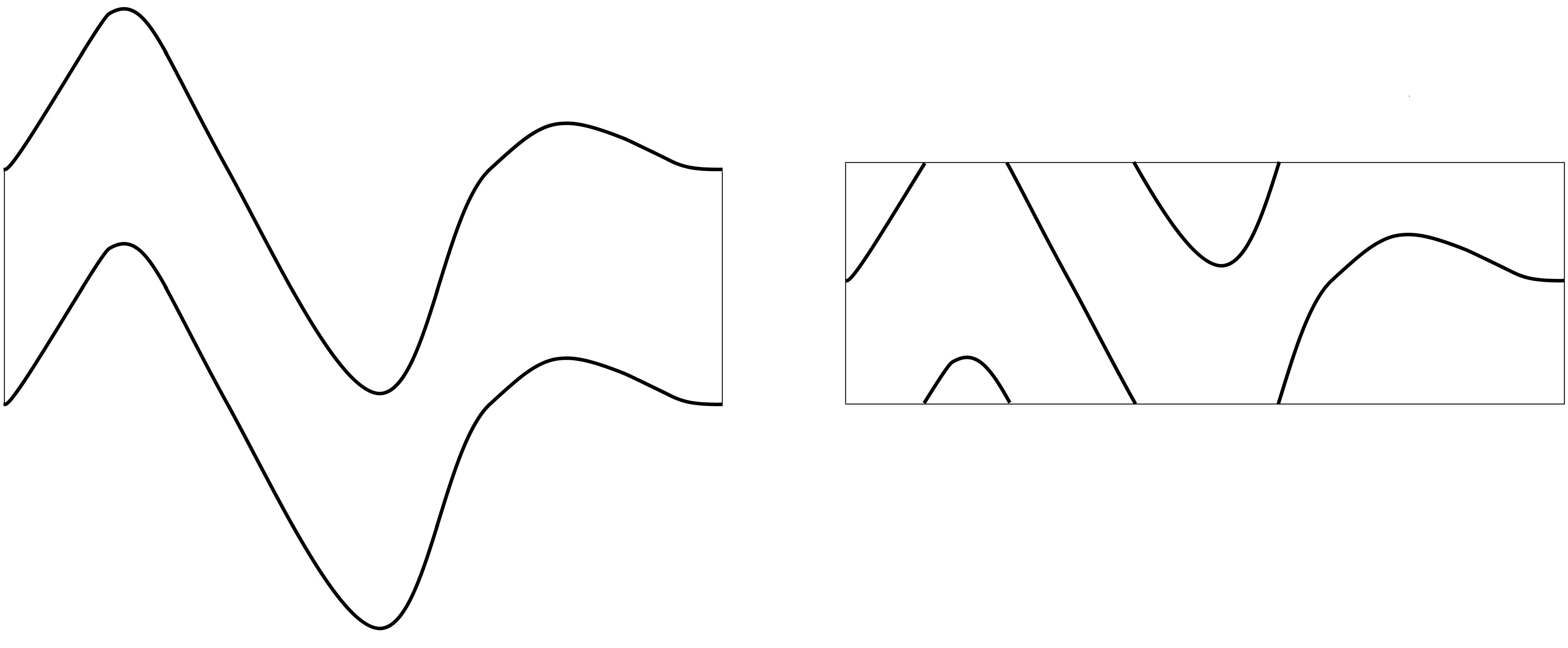}}
    \caption{$\Omega^\eta(t)$ on the left and $\Omega$ on the right.}
    \label{figure2}
\end{figure}

For a set\footnote{Here, $S$ will represents either one of the sets $(0,T)$, $\Gamma$, $\Omega$ or some of their products.} $S=(a_1,a_1+L_1)\times \dots \times (a_n,a_n+L_n)$ where  $L_1,...,L_n>0$ and $n\in \{1,2,3\}$, we introduce the spaces of differentiable periodic functions for $k \in \mathbb{N}_0 \cup \{\infty\}$
\begin{multline*}
   C_\#^k(S):=\{f \in C^k(\mathbb{R}^n): f(x_1,\ldots,x_n) = f(x_1+L_1,\ldots,x_n) =...=f(x_1,\ldots,x_n+L_n)\\ \text{ for all } (x_1,\ldots,x_n)\in \mathbb{R}^n \}.
\end{multline*}
We define Lebesgue and Sobolev function spaces for any $p,q\in[1,\infty]$, $k \in \mathbb{N}_0 \cup \{\infty\}$ as closures in the respective norms
\begin{eqnarray*}
    W_\#^{k,p}(S):=\overline{C_\#^\infty(S)}^{\|\cdot\|_{W^{k,p}(S)}}.
\end{eqnarray*}
In order to accommodate the boundary conditions \eqref{boundaryconditions1} we further introduce the spaces
\begin{equation*}
\begin{split}
    &C^k_{\#,0}(\Gamma):=\{\varphi\in C^k_\#(\Gamma): \varphi(0) 
    = 0 \}, \\
    &C^k_{\#,0}(\Gamma_T):=\{\varphi\in C^k_\#(\Gamma_T): \varphi(t,0) 
    = 0 \text{ for all } t \in \R \},
\end{split}
\end{equation*}
for $k \in \mathbb{N}_0 \cup \{\infty\}$, and the corresponding closure 
\begin{eqnarray*}
    W_{\#,0}^{k,p}(\Gamma):=\overline{C_{\#,0}^\infty(\Gamma)}^{\|\cdot\|_{W^{k,p}(\Gamma)}}.
\end{eqnarray*}
Finally, we define 
\begin{equation*}
\begin{split}
    &L_\#^p(0,T;W_\#^{1,q}(\Omega)):=\{f\in L_\#^p(0,T;L_\#^q(\Omega)): \nabla f\in L_\#^p(0,T;L_\#^q(\Omega))\}, \\
    &W_\#^{1,p}(0,T;L_\#^{q}(\Gamma)):=\{f\in L_\#^p(0,T;L_\#^q(\Gamma)): \partial_t f\in L_\#^p(0,T;L_\#^q(\Gamma))\}.
\end{split}
\end{equation*}
As usual, $H^k$ denotes Sobolev spaces $W^{k,2}$.
For a function $f\in C_\#^1(\Omega)$ and $\eta\in C^1_{\#,0}(\Gamma)$, we can define the Lagrangian trace on $\hat\Gamma^\eta$ as
\begin{equation*}
    \gamma_{|\hat{\Gamma}^\eta} f(x):=f(x,\hat\eta(x))
\end{equation*}
and then extend it to a linear and continuous operator $\gamma_{|\hat{\Gamma}^\eta}: H^1_\#(\Omega)\to H^{\frac12}_\#(\Gamma)$. Here $H^{\frac12}$ denotes the Sobolev-Slobodetskii space. Finally, we will denote the two-dimensional space variable $\bb{y} = (x,z)$.

\begin{mydef}[\textbf{Weak solution}]\label{weak:sol:def}
We say that $\rho \in L_\#^\infty(0,T; L_\#^\gamma(\Omega))$, $\bb{u}\in L_\#^2(0,T; H_\#^1(\Omega))$ and $\eta \in W_\#^{1,\infty}(0,T; L^2_\#(\Gamma))\cap L_\#^\infty(0,T; H_{\#}^2(\Gamma))\cap H_\#^1(0,T; H_{\#,0}^1(\Omega))$ is a weak solution to \eqref{structureeqs}-\eqref{eq:timeper} if:
\begin{enumerate}
    \item The kinematic coupling $\gamma_{|\hat{\Gamma}^\eta}\bb{u} = \eta_t \bb{e}_2$ holds on $\Gamma_T$.
    \item The renormalized continuity equation
\begin{equation}\label{reconteqweak}
 \int_{Q_T} \rho B(\rho)( \partial_t \varphi +\bb{u}\cdot \nabla \varphi)\di \bb{y} \dd t =\int_{Q_T} b(\rho)(\nabla\cdot \bb{u}) \varphi \di \bb{y} \dd t
\end{equation}
holds for all functions $\varphi \in C_\#^\infty(Q_T)$ and any $b\in L^\infty (0,\infty) \cap C[0,\infty)$ such that $b(0)=0$ with $B(\rho)=B(1)+\int_1^\rho \frac{b(z)}{z^2}dz$.
\item The coupled momentum equation
\begin{multline}
    \int_{Q_T} \rho \bb{u} \cdot\partial_t \boldsymbol\varphi \di \bb{y} \dd t + \int_{Q_T}(\rho \bb{u} \otimes \bb{u}):\nabla\boldsymbol\varphi \di \bb{y} \dd t +\int_{Q_T} \rho^\gamma (\nabla \cdot \boldsymbol\varphi) \di \bb{y} \dd t -  \int_{Q_T} \mathbb{S}( \nabla\bb{u}): \nabla \boldsymbol\varphi \di \bb{y} \dd t\\
    +\int_{\Gamma_T} \eta_t \psi_t \di x \dd t - \int_{\Gamma_T}\eta_{xx} \psi_{xx} \di x \dd t- \bint_{\Gamma_T} \eta_{tx} \psi_x  \di x \dd t   =  -\int_{\Gamma_T} f\psi \di x \dd t - \int_{Q_T} \rho \bb{F}\cdot \boldsymbol\varphi \di \bb{y} \dd t \label{momeqweak}
\end{multline}
holds for all $\boldsymbol\varphi \in C_{\#}^\infty(Q_T)$ 
and all 
$\psi\in C^\infty_{\#,0}(\Gamma_T)$ 
such that $\boldsymbol\varphi (t,x,\hat{\eta}(t,x))=\psi(t,x)\bb{e}_2$ on $\Gamma_T$.
\end{enumerate}
\end{mydef}

We note that the choice $b(\rho) = 0$ in \eqref{reconteqweak} recovers the standard weak formulation of the continuity equation. Our main result reads as follows.

\begin{thm}[\textbf{Main result}]\label{t:main}
Let $H,L,T,m_0>0$ be given and let $\gamma > 1$. Let $f\in L_\#^2(\Gamma_T)$ and $\bb{F}\in L_\#^2(0,T; L_\#^\infty(\Omega))$. Then, there exists at least one weak solution to \eqref{structureeqs}-\eqref{eq:timeper} in the sense of Definition \ref{weak:sol:def} such that
\begin{equation*}
    \int_{\Omega}\rho(t)\di \bb{y}=m_0
\end{equation*}
for almost all $t\in (0,T)$ and the energy inequality
\begin{eqnarray}
    &&-\int_{Q_T}\phi_t\left(\frac12  \rho|\bb{u}|^2 + \frac{1}{\gamma-1}\rho^\gamma \right) \di \bb{y} \dd t- \int_{\Gamma_T}\phi_t\left(\frac12 |\eta_t |^2 + \frac12 |\eta_{xx}  |^2\right)(t)\di x \dd t \nonumber \\
    &&\quad +\int_0^T \int_{\Omega}\phi\mathbb{S}(\nabla \bb{u}):\nabla \bb{u} \di \bb{y} \dd t +\int_0^T  \int_\Gamma \phi | \eta_{tx}|^2 \di x \dd t \nonumber\\
    &&\quad \leq \int_0^T  \int_\Gamma \phi f\eta_t \di x \dd t + \int_0^T \int_\Omega \phi\rho\bb{u}\cdot\bb{F} \di \bb{y} \dd t \label{eq:EE1}
\end{eqnarray}
holds for all $\phi \in C_\#^\infty(0,T)$, $\phi \geq 0$. Moreover,
\begin{multline}\label{eq:EEbound}
  \sup_{t \in (0,T)}\Big[ \int_{\Omega} \Big( \frac{1}{2} \rho |\bb{u}|^2 + \frac{1}{\gamma-1}\rho^\gamma\Big) \di \bb{y} + \int_\Gamma \Big(\frac{1}{2}|\eta_t|^2 + \frac{1}{2} |\eta_{xx}|^2\Big)\di x\Big](t)\\
  +\int_{Q_T}\mathbb{S}(\nabla \bb{u}):\nabla \bb{u} \di \bb{y}\dd t + \int_{\Gamma_T}|\eta_{tx}|^2  \di x\dd t \leq C(f,\bb{F},\Omega,m_0). 
\end{multline}

\end{thm}

\begin{rem}[Strategy of the proof]
The proof of this theorem is based on a  four-level approximation scheme. Following the approach from \cite{trwa} (see also \cite{MaMuNeRoTr}), we decouple the coupled momentum equation to the fluid momentum equation and the structure momentum equation by penalizing the kinematic coupling condition $\eqref{kinc}$. This allows us to deal with these equations separately. Then, we choose to span the fluid velocity and the structure displacement in finite time-space bases, as it was done in \cite{FMNP} (note that this is in contrast with the fixed-point approach which was used in \cite{FNPS,MS}). Finally, as it is standard in the theory of compressible Navier-Stokes equations, artificial diffusion is added to the fluid continuity equation and the artificial pressure is added to the fluid momentum equation. Several other terms are also added due to technical reasons. In order to obtain a weak solution, there are four limits that are performed, each of them being based on estimates that significantly differ from a limit to limit due to their high sensitivity to the approximation parameters. Unlike the initial value problem, we need to additionally ensure that the energy inequality of the form $\eqref{eq:EE1}$ is satisfied at each approximation level to obtain some important estimates, and for this we need to prove the convergence of the structure kinetic and elastic energies in each of the limits. This part is based on improved structure displacement estimates from \cite{MuhaSch}, adapted to our framework similarly as in \cite{Tr}.
\end{rem}

\begin{rem}\label{rem:domains}
Throughout the proof, we will work with formulations of the problem both on $\Omega$ and on $\Omega^\eta(t)$. As both the fluid velocity $\bb{u}$ and density $\rho$ can be represented on $\Omega^\eta(t)$ equivalently, we keep the same notation for $\bb{u}$ and $\rho$ whenever we shift to the domain $\Omega^\eta(t)$. Let us point out that $\bb{u}$ is continuous on $\hat{\Gamma}^\eta(t)$ so $\|\bb{u}\|_{W^{1,p}(\Omega^\eta(t))}= \|\bb{u}\|_{W^{1,p}(\Omega)}$ for any $p\in [1,\infty]$, while $\rho$ may have a jump on $\hat{\Gamma}^\eta(t)$ so we use $\|\rho\|_{L^{p}(\Omega^\eta(t))}= \|\rho\|_{L^{p}(\Omega)}$ for $p\in [1,\infty]$ only.
\end{rem}

\section{Discussion and literature overview}

The mathematical theory of the interaction problems between incompressible viscous fluids and thin elastic structures (plates or shells) has started with results of Beirao da Veiga \cite{veiga} and Grandmont et al. \cite{grandmont3,grandmont4}, and continued to develop in the last two decades, see \cite{LeRu14,MuhaSch,CanicMuha,muhacanic2,trwa,grandmont5,unrestricted} for the existence of weak solutions, \cite{AbelsLiu1, AbelsLiu2,
grandmont1,grandmont2,MR4189724,Avalos1,grandmont5,Leq,Leq1} for the existence of strong solutions and \cite{continuous,Sebastian} for uniqueness.  Theory involving compressible viscous fluids interacting with plates and shells on the other started quite recently with the result of Schwarzacher and Breit \cite{Breit}, and continued with \cite{Tri1} where weak solution was obtained for an interaction between a compressible viscous fluid and a nonlinear thermoelastic plate. Local in time regular solutions were constructed in \cite{sourav,RoyMaity}, while the weak-strong uniqueness for such problems was studied in \cite{Tr}. In the case of heat-conducting fluids, interaction with an elastic plate was considered in \cite{Breit2} where a weak solution was constructed which satisfies the energy equality, and an interaction with a viscoelastic plate was considered in \cite{MaityTakahahi} where the strong solution with maximal regularity was constructed. The interaction of heat-conducting fluids and thermoelastic shells with heat exchange was studied in \cite{MaMuNeRoTr}, where a weak solution was constructed. The case of mixture with elastic structure was studied in \cite{KMN}. A semigroup approach to wellposedness of the problem of interaction of a linearized compressible fluid with an elastic boundary was presented in \cite{Avalos2}. Finally, local in time regular solutions to the interaction problems between 3D elastic solids and fluids were obtained in \cite{CS05, CS06,KT12,RV14,BG17}, while weak solutions were constructed in \cite{benevsova2020variational,breit2021compressible}. We also refer the reader to a very recent result \cite{contact} where such a problem with allowed contact was studied.

With all this in mind, little attention has been given to time-periodic solutions, or more precisely, to the question when the fluid-structure interaction model has a periodic behaviour under periodic forcing. Indeed, this question is of big importance, since many models tend to show periodic behavior. For example, heart beats and air flow through trachea are both periodic. Therefore, one can naturally ask, under what condition we can expect such models to behave periodically? This was first studied by Casanova for an interaction problem between a viscoelastic beam and an incompressible fluid \cite{casanova} in the framework of strong solutions. Quite recently, Schwarzacher and M\^{i}ndril\v{a} studied the interaction of a linear Koiter shell with an incompressible viscous fluid and obtained existence of a weak solution with a closed rigid boundary with no-slip condition in \cite{MS} and a dynamic pressure boundary condition in \cite{MS1}. Finally, concerning the purely fluid system, the time-periodic weak solutions to the compressible Navier-Stokes system on a fixed domain were constructed in \cite{FNPS} for isentropic flows and in \cite{FMNP} for the full Navier-Stokes-Fourier system.

The main goal of this paper is to tackle this issue in the case when the fluid is compressible. This brings many challenges which do not exist in the incompressible case. The main challenge in the compressible viscous fluid theory is dealing with pressure and our case is no different. The estimates based on Bogovskii operator for the pressure are very sensitive to the shape of the domain (and thus on deformations of the beam) and many other factors including dimension. This directly results in limitations in our result, i.e. the dimension of the fluid is two, the beam is visoelastic and the fluid domain is periodic in horizontal and vertical direction which a priori excludes contact for the beam.

The paper is organized as follows. In Section 4 we present a way to obtain a priori estimates assuming the solution is sufficiently smooth. This procedure is split into several steps. In Section 5 we present the approximation scheme used in the proof of Theorem \ref{t:main} and prove the existence of a solution to the approximated system. In Section 6 we pass to the limit in the number of time basis functions $m \to \infty$ and present uniform estimates for the arising solution independent of $n$. In Section 7 we pass to the limit in the number of spatial basis functions $n \to \infty$, deduce uniform bounds independent of $\ep$ and introduce the coupled momentum equation. In Section 8, we perform the limit with the penalization and artificial density diffusion parameter $\ep \to 0$  and deduce uniform bounds independent of $\delta$. Finally, in Section 9 we pass to the limit with $\delta \to 0$, thus removing the artificial pressure term and finishing the proof of Theorem \ref{t:main}.

\section{A priori estimates for smooth solutions}\label{aprioriest}
Before we start, let us introduce the energy associated to the studied system as 
\begin{equation*}
    E(t) := \int_{\Omega}\left(\frac12  \rho|\bb{u}|^2 + \frac{1}{\gamma-1}\rho^\gamma \right)(t)\di \bb{y} + \int_{\Gamma}\left(\frac12 |\eta_t |^2 + \frac12 |\eta_{xx}  |^2\right)(t)\di x
\end{equation*}
and we emphasize that replacing $\Omega$ with $\Omega^\eta(t)$ yields the same quantity, see Remark \ref{rem:domains}. Further, we denote 
\begin{equation*}
    \mathcal{E} :=  \sup\limits_{(0,T)} E.
\end{equation*}
The goal of this section is to show that smooth solutions to the problem \eqref{structureeqs}-\eqref{eq:timeper} satisfies the inequality $\eqref{eq:EEbound}$. This will serve as base in the forthcoming sections, where approximate problems with similar properties will be studied. We note that since we assume in this section that the solution is smooth, we are allowed to consider unbounded functions $b$ in \eqref{reconteqweak}.


\subsection{Part I - estimates of $\nabla\bb{u}$ and $\eta_{tx}$}\label{partI}
In order to obtain the estimates, we sum up \eqref{reconteqweak} with $b(\rho)=\rho^\gamma$ and $\varphi=1$, \eqref{reconteqweak} with $b(\rho)=0$ and $\varphi=\frac12|\bb{u}|^2$ and \eqref{momeqweak} with $(\boldsymbol\varphi,\psi)=(\bb{u},\eta_t)$ to obtain
\begin{equation*}
     \int_{Q_T} \mathbb{S}(\nabla \bb{u}):\nabla \bb{u}\di \bb{y}\dd t + \int_{\Gamma_T} |\eta_{tx}|^2\di x\dd t=\int_{\Gamma_T}f\eta_t\di x\dd t+\int_{Q_T} \rho\bb{u}\cdot \bb{F}\di \bb{y}\dd t
\end{equation*}
and thus
\begin{multline*}
    \int_{Q_T} \mathbb{S}(\nabla \bb{u}):\nabla \bb{u} \di \bb{y} \dd t+
    c(L)\|\eta_t \|_{L^2(0,T;H^1(\Gamma))}^2 \\
    \leq \int_{Q_T} \mathbb{S}(\nabla \bb{u}):\nabla \bb{u} \di \bb{y} \dd t + \int_{\Gamma_T} |\eta_{tx}|^2 \di x \dd t 
    = \int_{\Gamma_T}f\eta_t \di x \dd t +\int_{Q_T} \rho\bb{u}\cdot \bb{F} \di \bb{y} \dd t\\
    \leq \|{f}\|_{L^2(\Gamma_T)}  \|\eta_t \|_{L^2(\Gamma_T)} + \|\rho\|_{L^\infty(0,T; L^p(\Omega))}\|\bb{u}\|_{L^2(0,T; L^q(\Omega))} \| \bb{F} \|_{L^2(0,T;L^\infty(\Omega))}\\
    \leq C(f,L) + \frac{c(L)}2\| \eta_t \|_{L^2(0,T; H^1(\Gamma))}^2+C(\bb{F})\|\rho\|_{L^\infty(0,T; L^p(\Omega))} \|\bb{u}\|_{L^2(0,T; L^q(\Omega))}
\end{multline*}
for any $p>1$ and $q=\frac{p}{p-1}$ by the Poincar\'e inequality for $\eta$. We have just deduced that
\begin{equation} 
    \int_{Q_T} \mathbb{S}(\nabla \bb{u}):\nabla \bb{u} \di \bb{y} \dd t +
    \|\eta_t \|_{L^2(0,T;H^1(\Gamma))}^2
    \leq C+C\|\rho\|_{L^\infty(0,T; L^p(\Omega))} \|\bb{u}\|_{L^2(0,T;L^q(\Omega))}. \label{eq:AP1}
\end{equation}
From here onward, we omit the dependence of constants on $\Omega,f,\bb{F}$, since they are given and do not depend on functions $\rho,\bb{u},\eta$. 

Next, we shift to the moving domain $\Omega^\eta(t)$ given in \eqref{eqv:domain}. We have
\begin{equation*}
    \|\eta_t \bb{e}_2\|_{L^2(0,T; H^1(\Omega^\eta(t)))} = 2H \|\eta_t \|_{L^2(0,T; H^1(\Gamma))}.
\end{equation*}
Due to the kinematic coupling, we have that $\bb{u}-\eta \bb{e}_2=0$ on $\Gamma^\eta(t)$ and $\Gamma^\eta(t)+2H$, so by using the Korn identity on $\Omega^\eta(t)$
\begin{multline*}
    \|\nabla\bb{u} - \nabla(\eta_t \bb{e}_2)\|_{L^2(Q_T^\eta)}^2 + \| \nabla\cdot(\bb{u}-\eta_t\bb{e}_2)\|_{L^2(Q_T^\eta)}^2 = 2\| \mathbb{D}(\bb{u} -\eta_t\bb{e}_2)\|_{L^2(Q_T^\eta)}^2 \\
    \leq C\|\mathbb{S}(\nabla\bb{u}-\nabla(\eta_t\bb{e}_2)\|_{L^2(Q_T^\eta)}^2 \leq C \left( \int_{Q_T^\eta} \mathbb{S}(\nabla \bb{u}):\nabla \bb{u} \di \bb{y} \dd t +
    \|\eta_t \|_{L^2(0,T;H^1(\Gamma))}^2 \right), 
\end{multline*}
where $C$ only depends on $\mu,\zeta$. The Poincar\'{e} inequality yields
\begin{equation*}
    \|\bb{u}-\eta_t\bb{e}_2\|_{H^1(\Omega^\eta(t))}^2 \leq C\|\nabla\bb{u}-\nabla(\eta_t\bb{e}_2)\|_{L^2(\Omega^\eta(t))}^2.
\end{equation*}
Note that the constant $C$ is independent of $\eta$ -- this follows directly from the proof of the inequality for the steady domain \cite[Theorem 6.30]{AdFo}.
We use
\begin{equation*}
    \|\eta_t\|_{L^\infty(\Gamma)} \leq C \|\eta_t\|_{H^1(\Gamma)}
\end{equation*}
and $\bb{u}-\eta_t\bb{e}_2=0$ on $\Gamma^\eta(t)\cup\Gamma^\eta(t)+2H$ to conclude
\begin{multline}
    \|\bb{u}\|_{L^2(0,T; L^q(\Omega^\eta(t)))}^2\leq 2\|\bb{u}-\eta_t\bb{e}_2\|_{L^2(0,T; L^q(\Omega^\eta(t)))}^2+2\|\eta_t\bb{e}_2 \|_{L^2(0,T; L^q(\Omega^\eta(t)))}^2\\
    \leq C\|\bb{u}-\eta_t\bb{e}_2\|_{L^2(0,T; H^1(\Omega^\eta(t)))}^2+ C\|\eta_t\bb{e}_2 \|_{L^2(0,T; H^1(\Omega^\eta(t)))}^2\\
    \leq C \int_{Q_T^\eta} \mathbb{S}(\nabla \bb{u}):\nabla \bb{u} \di \bb{y} \dd t + C\|\eta_t\|_{L^2(0,T; H^1(\Gamma))}^2 \label{eq:AP2}
\end{multline}
for any $1<q<\infty$. 
We set\begin{equation}\label{epcond}
    \kappa := \min\left\{\frac1{20}, \frac{(\gamma-1)}{5\gamma}, \frac{1}{5(\gamma-1)}\right\}, 
\end{equation}
and
\begin{equation*}
    \overline{p} = \overline{p}(\kappa) := \frac{2\gamma^2}{2\gamma^2 - \kappa(\gamma-1)},
\end{equation*}    
so we have
\begin{equation*}
    \overline{\theta}:=\frac{\gamma(\overline{p}-1)}{\overline{p}(\gamma-1)} = \frac\kappa{2\gamma}.
\end{equation*}
Then for any $1 < p < \overline{p}$ we have for some $\theta < \overline{\theta}$
\begin{equation}
    \|\rho\|_{L^\infty(0,T; L^p(\Omega^\eta(t)))}\leq \|\rho\|_{L^\infty(0,T; L^1(\Omega^\eta(t)))}^{1-\theta} \|\rho\|_{L^\infty(0,T; L^\gamma(\Omega^\eta(t)))}^{\theta} \leq C m_0^{1-\theta} \mathcal{E}^{\gamma\theta}\leq C(1 + \mathcal{E}^{\frac{\kappa}{2}}).\label{rho:est}
\end{equation}
Since 
\begin{equation*}
    \|\rho\|_{L^\infty(0,T; L^p(\Omega^\eta(t)))}=\|\rho\|_{L^\infty(0,T; L^p(\Omega))}, \quad \int_{Q_T^\eta} \mathbb{S}(\nabla \bb{u}):\nabla \bb{u}\di \bb{y} \dd t = \int_{Q_T} \mathbb{S}(\nabla \bb{u}):\nabla \bb{u}\di \bb{y} \dd t,
\end{equation*}
the inequalities \eqref{eq:AP1}, \eqref{eq:AP2} and \eqref{rho:est} yield
\begin{equation}\label{diss:est}
   \| \bb{u}\|_{L^2(0,T;H^1(\Omega))}^2+
    \|\eta_t \|_{L^2(0,T;H^1(\Gamma))}^2 \leq C(\kappa)(1+\mathcal{E}^\kappa),
\end{equation}
for the original domain, and consequently
\begin{equation}\label{u:est}
    \|\bb{u} \|_{L^2(0,T;L^q(\Omega))}^2 \leq C(\kappa,q)(1+\mathcal{E}^\kappa)
\end{equation}
for all $q > 1$.


\subsection{Part II - circular estimates}\label{partII}
In order to deduce the energy inequality, we sum up \eqref{reconteqweak} with $b(\rho)=\rho^\gamma$ and $\varphi=\chi_{[s,t]}$, \eqref{reconteqweak} with $b(\rho)=0$ and $\varphi=\chi_{[s,t]}\frac12|\bb{u}|^2$ and \eqref{momeqweak} with $(\boldsymbol\varphi,\psi)=(\chi_{[s,t]}\bb{u},\chi_{[s,t]}\eta_t)$ to obtain
\begin{multline*}
    E(t) +\int_s^t\int_{\Omega} \mathbb{S}(\nabla \bb{u}):\nabla \bb{u}\di \bb{y} \dd \tau + \int_s^t\int_{\Gamma} |\eta_{tx}|^2\di x \dd\tau
    \\
    = E(s) + \int_s^t \int_{\Gamma} f \eta_t  \di x \dd\tau+ \int_s^t\int_{\Omega} \rho\bb{u}\cdot \bb{F}\di \bb{y} \dd \tau\\
    \leq E(s)+C(\kappa)(1+\mathcal{E}^\kappa) 
    \leq E(s)+C(\kappa)+\kappa \mathcal{E}
\end{multline*}
by \eqref{rho:est}, \eqref{diss:est}, \eqref{u:est} and the Young inequality. We integrate again over $(0,T)$ with respect to variable $s$ and then we take a supremum in the variable $t$ over $(0,T)$ on the left hand side to obtain
\begin{equation}\label{eq:CIRC1}
  \mathcal{E} \leq C_0\left(1+\int_0^T E(s)\di s\right).
\end{equation}
The constant $C_0$ depends on the choice of $\kappa$, however we recall that $\kappa$ is already fixed. Our goal in the remaining part of the estimates is to show 
\begin{equation*}
  \int_0^T E(s)\di s \leq \delta_0\mathcal{E} + C(\delta_0)
\end{equation*}
for some $\delta_0\in (0,\frac1{C_0})$.

\subsection{Part III - estimate of $\eta_{xx}$}\label{eta:est:sec}
In this section we need the following interpolation inequality.
\begin{lem}\label{l:interpol}
Let $g \in H^1(0,T;L^2(\Gamma)) \cap L^2(0,T;H^1(\Gamma))$. Then for any $\alpha \in (0,1)$ it holds $$
g \in H^\alpha(0,T;H^{1-\alpha}(\Gamma))
$$ 
and there exists a constant $C > 0$ independent of $g$ such that
\begin{equation*}
    \|g\|_{H^\alpha(0,T;H^{1-\alpha}(\Gamma))} \leq C \|g\|_{H^1(0,T;L^2(\Gamma))}^\alpha \|g\|_{L^2(0,T;H^1(\Gamma))}^{1-\alpha}.
\end{equation*}
\end{lem}
\begin{proof}
First, note that $g$ can easily be extended to $\mathbb{R}^2$ (also denoted as $g$) so that
\begin{eqnarray*}
    \|g\|_{H^1(\mathbb{R};L^2(\mathbb{R}))} \leq C \|g\|_{H^1(0,T;L^2(\Gamma))}, \quad  \|g\|_{L^2(\mathbb{R};H^1(\mathbb{R}))} \leq C \|g\|_{L^2(0,T;H^1(\Gamma))}.
\end{eqnarray*}
Denote as $\mathcal{F}_{t}$, $\mathcal{F}_x$ and $\mathcal{F}_{t,x}$ the Fourier transform w.r.t. variables $t$ and $x$ and both $t,x$, respectively. One has:
\begin{eqnarray*}
     \|g\|_{H^\alpha(\mathbb{R};H^{1-\alpha}(\mathbb{R}))}^2 &\leq&  C\int_{\mathbb{R}} (1+\sigma^2)^\alpha || \mathcal{F}_t(g)||_{H^{1-\alpha}(\mathbb{R})}^2 ~d\sigma \\
     &\leq& C\int_{\mathbb{R}} (1+\sigma^2)^\alpha \int_\mathbb{R} (1+\xi^2)^{1-\alpha} |\mathcal{F}_x(\mathcal{F}_t(g))|^2~ d\xi  d\sigma \\
     &=& C\int_{\mathbb{R}^2}(1+\sigma^2)^\alpha (1+\xi^2)^{1-\alpha} |\mathcal{F}_{t,x} (g)|^2 ~ d\xi  d\sigma\\
     &\leq&  C\left(\int_{\mathbb{R}^2}(1+\sigma^2) |\mathcal{F}_{t,x} (g)|^2 ~ d\xi  d\sigma  \right)^{2\alpha } \left(\int_{\mathbb{R}^2}(1+\xi^2) |\mathcal{F}_{t,x} (g)|^2 ~ d\xi  d\sigma  \right)^{2(1-\alpha)} \\
     &=& C\|g\|_{H^1(\mathbb{R};L^2(\mathbb{R}))}^{2\alpha}\|g\|_{L^2(\mathbb{R};H^1(\mathbb{R}))}^{2(1-\alpha)},
\end{eqnarray*}
where we used H\"{o}lder's inequality with indexes $p=\frac1\alpha$ and $q=\frac1{1-\alpha}$.

\end{proof}

We use test functions $(\boldsymbol\varphi, \psi) = (\eta\bb{e}_2,\eta)$ in \eqref{momeqweak}, we observe  that $\nabla\cdot(\eta\bb{e}_2) = 0$ and 
\begin{equation*}
    \int_{\Gamma_T}  \eta_{tx} \eta_x\di x \dd t= \frac12 \int_{\Gamma_T}  (\eta_{x}^2)_t\di x \dd t = 0.
\end{equation*}
Consequently,
\begin{multline}\label{eq:est_etaxx1}
  \|\eta_{xx}\|^2_{L^2(\Gamma_T)} =  \int_{\Gamma_T}| \eta_{xx}  |^2 \di x \dd t \\
   = \int_{Q_T} \rho \bb{u} \cdot \eta_t  \bb{e}_2\di \bb{y} \dd t + \int_{Q_T} \rho \bb{u}\otimes \bb{u} :\nabla(\eta  \bb{e}_2)\di \bb{y} \dd t - \int_{Q_T} \mathbb{S}(\nabla \bb{u}):\nabla (\eta \bb{e}_2)\di \bb{y} \dd t +\int_{Q_T}\rho\eta \bb{e}_2\cdot \bb{F}\di \bb{y} \dd t \\
  \qquad + \int_{\Gamma_T} |\eta_t |^2\di x \dd t  +\int_{\Gamma_T} f  \eta \di x \dd t.
\end{multline}
 We fix $1<p<\overline{p}$, denote $q = \frac{p}{p-1}$ and estimate the terms on the right hand side as follows. First, 
\begin{equation*}
    \left|\int_{Q_T} \rho \bb{u} \cdot \eta_t  \bb{e}_2\di \bb{y} \dd t\right| \leq C\|\rho\|_{L^\infty(0,T; L^p(\Omega))}\|\bb{u}\|_{L^2(0,T;L^q(\Omega))} \|  \eta_t\|_{L^2(0,T;L^\infty(\Gamma))}\leq C(\kappa)\left(1+\mathcal{E}^{\frac{3\kappa}2}\right)
\end{equation*}
by using Sobolev embedding, \eqref{rho:est}, \eqref{diss:est} and \eqref{u:est}. In order to estimate the convective term, we utilize the following estimate
\begin{multline}
    \|\eta_x\|_{L^\infty(0,T;L^{3q}(\Gamma))}\leq C\|\eta_{x}\|_{H^{\frac12+\delta}(0,T;H^{\frac12-\delta}(\Gamma))} \leq  C\|\eta_{x}\|_{H^1(0,T; L^2(\Gamma))}^{\frac12+\delta}\|\eta_{x}\|_{L^2(0,T; H^{1}(\Gamma))}^{\frac12-\delta}\\
    \leq C\big(\|\eta_{x}\|_{L^2(0,T; L^2(\Gamma))}^{\frac12+\delta}+\|\eta_{tx}\|_{L^2(0,T; L^2(\Gamma))}^{\frac12+\delta}\big)\|\eta_{x}\|_{L^2(0,T; H^1(\Gamma))}^{\frac12-\delta}\\
    =C \|\eta_{x}\|_{L^2(0,T; H^1(\Gamma))}+C\|\eta_{tx}\|_{L^2(0,T; L^2(\Gamma))}^{\frac12+\delta}\|\eta_{x}\|_{L^2(0,T; H^1(\Gamma))}^{\frac12-\delta}\\
    \leq C\big(\|\eta_{x}\|_{L^2(0,T; H^1(\Gamma))} + \|\eta_{tx}\|_{L^2(0,T; L^2(\Gamma))}\big). \label{eq:int_est}
\end{multline}
Here $\delta>0$ is sufficiently small, we have used Sobolev embedding, Lemma \ref{l:interpol} and the Young inequality for exponents $(\frac{1}{2}+\delta)^{-1}$ and $(\frac{1}{2}-\delta)^{-1}$. We use this estimate to write 
\begin{multline*}
    \left|\int_{Q_T} \rho \bb{u}\otimes \bb{u} :\nabla(\eta  \bb{e}_2)\di \bb{y} \dd t \right| \leq C\|\rho\|_{L^\infty(0,T; L^p(\Omega)}\|\bb{u}\|_{L^2(0,T;L^{3q}(\Omega))}^2 \|\eta_x\|_{L^\infty(0,T;L^{3q}(\Gamma))} \\
    \leq C(\kappa)(1+\mathcal{E}^{\frac{3\kappa}2} )    \left(\|\eta_{x}\|_{L^2(0,T; H^1(\Gamma))} + \|\eta_{tx}\|_{L^2(0,T; L^2(\Gamma))}\right)
    \leq C(\kappa)\left( 1+\mathcal{E}^{3\kappa}\right) + \frac18 \|\eta_{xx}\|^2_{L^2(\Gamma_T)},
\end{multline*}
where we have used again \eqref{rho:est}, \eqref{diss:est}, \eqref{u:est},  and the Young inequality. The viscous term is estimated by
\begin{multline*}
    \left|\int_{Q_T} \mathbb{S}(\nabla \bb{u}):\nabla (\eta \bb{e}_2)\di \bb{y} \dd t\right| \leq C\|\mathbb{S}(\nabla \bb{u})\|_{L^2(Q_T)} \|\eta_x \|_{L^2(0,T;L^2(\Gamma))} \\
    \leq C\|\mathbb{S}(\nabla \bb{u})\|_{L^2(Q_T)}^2+\frac{1}{8}\|\eta_{xx}  \|_{L^2(\Gamma_T)}^2 \leq  C(\kappa)\left( 1+\mathcal{E}^{\kappa}\right) + \frac{1}{8}\|\eta_{xx}  \|_{L^2(\Gamma_T)}^2
\end{multline*}
using \eqref{diss:est}. We also use \eqref{diss:est} directly to estimate
\begin{equation*}
    \int_{\Gamma_T} |\eta_t|^2\di x \dd t \leq C(\kappa)\left( 1+\mathcal{E}^{\kappa}\right).
\end{equation*}
Finally, 
\begin{equation*}
    \left|\int_{Q_T} \rho\eta\bb{e}_2\cdot\bb{F}\di \bb{y} \dd t\right| \leq C\|\rho\|_{L^\infty(0,T; L^1(\Omega))}\|\eta\|_{L^2(0,T;L^\infty(\Gamma))}\|\bb{F}\|_{L^2(0,T;L^\infty(\Omega))}\leq C + \frac18 \|\eta_{xx}\|_{L^2(\Gamma_T)}^2
\end{equation*}
and
\begin{equation*}   
    \left|\int_{\Gamma_T} f \eta \di x \dd t \right| \leq \|f\|_{L^\infty(\Gamma_T)}\|\eta\|_{L^1(\Gamma_T)} \leq C + \frac18 \|\eta_{xx}\|_{L^2(\Gamma_T)}^2
\end{equation*}
by using the Poincar\'e inequality twice together with the boundary condition \eqref{boundaryconditions1}. All the estimates together with \eqref{eq:est_etaxx1} yield
\begin{equation}\label{eq:int_est_final}
    \int_{\Gamma_T}| \eta_{xx}  |^2 \di x \dd t\leq  C(\kappa)(1+\mathcal{E}^{3\kappa}).
\end{equation}


\subsection{Part IV - density/pressure estimates}\label{Bog:section}

Denote the Bogovskii operator as $\mathcal{B}_\Omega: L_0^p(\Omega) \to W_0^{1,p}(\Omega)$. This operator satisfies
\begin{equation*}
    \nabla \cdot \mathcal{B}_\Omega[f] = f,
\end{equation*}
where $L_0^p(\Omega):=\{f\in L^p(\Omega): \int_\Omega f = 0\}$ and $W^{1,p}_0(\Omega):=\{f\in W^{1,p}(\Omega): f{\restriction}_{\partial \Omega}=0\}$. Moreover, 
\begin{equation*}
    \|\mathcal{B}_\Omega[f]\|_{W^{1,p}(\Omega)} \leq C \|f\|_{L^p(\Omega)}.
\end{equation*}
Throughout the rest of this section, we will repeatedly use the following estimate. For $0<\alpha<\frac12$, we have
\begin{equation}\label{eq:Bogovskii_infty}
    \left\|\mathcal{B}_\Omega\left[\rho^\alpha-\int_{\Omega}\rho^\alpha\di \bb{y}\right]\right\|_{L^\infty(\Omega)} \leq C\left\|\mathcal{B}_\Omega\left[\rho^\alpha-\int_{\Omega}\rho^\alpha\di \bb{y}\right]\right\|_{W^{1,\frac{1}{\alpha}}(\Omega)} \leq C \| \rho^\alpha\|_{L^{\frac1\alpha}(\Omega)} =C m_0^\alpha.
\end{equation}
We cannot use $\mathcal{B}_\Omega[\rho^\alpha-\int_{\Omega}\rho^\alpha]$ as a test function $\boldsymbol\varphi$ in \eqref{momeqweak} since its trace on $\Gamma^\eta$ is not regular enough in general. Therefore, we split the procedure into estimates near the viscoelastic structure and estimates in the interior of the fluid domain. To this end we fix $0<h<\frac{{H}}{2}$ and we emphasize that constants appearing in the calculations below may depend on $h$. \\

 We shift to the moving domain $\Omega^\eta(t)$ and we deal with the  interior estimates first. Note that the function $\mathcal{B}_\Omega\left[\rho^\alpha-\int_{\Omega}\rho^\alpha\di \bb{y} \right]$ shifted to $\Omega^\eta(t)$ does not vanish on its boundary $\Gamma^\eta(t)$ and $\Gamma^\eta(t)+2H$. For that reason, we define a cut-off function
\begin{equation*}
    \phi_h(t,x,z):=\begin{cases}
    \frac{z-\eta(t,x)}{h}, &\quad \text{ for } \eta(t,x)<z<\eta(t,x)+h,\\
    1, &\quad \text{ for } \eta(t,x)+h<z<\eta(t,x)+2{H}-h,\\
    \frac{2{H}+\eta(t,x)-z}{h}, &\quad \text{ for } \eta(t,x)+2{H}-h<z<\eta(t,x)+2{H},
    \end{cases}
\end{equation*}
and 
\begin{equation}\label{eq:varphi_h_def}
    \boldsymbol\varphi_h:=\phi_h\mathcal{B}_\Omega\left[\rho^\alpha-\int_{\Omega}\rho^\alpha\di \bb{y} \right],
\end{equation}
where 
\begin{equation*}
0<\alpha:=\min\left\{\frac25,\frac{\gamma-1}{2}\right\}
\end{equation*}
is fixed from now on. We emphasize that this choice of $\alpha$ ensures $\alpha < \frac12$, so we can use the estimate \eqref{eq:Bogovskii_infty}. Moreover due to \eqref{epcond} it holds
\begin{equation}\label{eq:alpha1}
   \frac32 \kappa(\gamma-1) < \alpha < \gamma - 1-\kappa\gamma,
\end{equation}
which will be important later.

We test the coupled momentum equation \eqref{momeqweak} by $(\boldsymbol\varphi_h,0)$ to obtain
\begin{multline}\label{eq:mainBog}
    \int_{Q_T^\eta} \rho^{\gamma+\alpha}\phi_h\di \bb{y} \dd t    = \int_{Q_T^\eta} \rho^{\gamma} \left(\int_{\Omega^\eta(t)} \rho^\alpha(t)\di \bb{y} \right)\phi_h\di \bb{y}\dd t\\
    -     \int_{Q_T^\eta} \rho^{\gamma}\left(\mathcal{B}_\Omega\left[\rho^\alpha-\int_{\Omega}\rho^\alpha\di \bb{y}\right] \cdot \nabla\phi_h \right)\di \bb{y} \dd t 
    - \int_{Q_T^\eta} \rho \bb{u}\cdot \partial_t \boldsymbol\varphi_h \di \bb{y} \dd t\\
     - \int_{Q_T^\eta} \rho \bb{u}\otimes \bb{u} : \nabla \boldsymbol\varphi_h \di \bb{y} \dd t + \int_{Q_T^\eta} \mathbb{S}(\nabla \bb{u}): \nabla \boldsymbol\varphi_h\di \bb{y} \dd t - \int_{Q_T^\eta} \rho \bb{F} \cdot \boldsymbol\varphi_h\di \bb{y} \dd t.
\end{multline}
We proceed to bound the terms on the right-hand side. Notice that
\begin{equation*}
    \int_{\Omega^\eta(t)} \rho^\alpha(t)\di \bb{y} \leq \left(\int_{\Omega^\eta(t)} \rho(t)\di \bb{y}\right)^\alpha |\Omega^\eta(t)|^{1-\alpha} \leq Cm_0^\alpha
\end{equation*}
and therefore
\begin{equation}\label{eq:Bogest11}
    \int_{Q_T^\eta} \rho^{\gamma} \left(\int_{\Omega^\eta(t)} \rho^\alpha(t)\di \bb{y}\right)\phi_h \di \bb{y} \dd t\leq C\mathcal{E} m_0^\alpha.
\end{equation}
Moreover,
\begin{multline}\label{eq:Bogest12}
    \left|\int_{Q_T^\eta} \rho^{\gamma}\mathcal{B}_\Omega\left[\rho^\alpha-\int_{\Omega}\rho^\alpha\di \bb{y}\right] \cdot \nabla\phi_h \right|\di \bb{y} \dd t\\
    \leq C \int_0^T  \|\rho^\gamma\|_{L^1(\Omega^\eta(t))} \left\|\mathcal{B}_\Omega\left[\rho^\alpha-\int_{\Omega}\rho^\alpha\di x\right]\right\|_{L^\infty(\Omega)} (1 + \| \eta_x \|_{L^\infty(\Gamma)})\di t \\
    \leq C \|\rho^\gamma\|_{L^\infty(0,T;L^1(\Omega^\eta(t)))}m^\alpha \left(1 + \| \eta\|_{L^2(0,T;H^{2}(\Gamma))}\right) 
     \leq  C(\kappa)\left(\mathcal{E}^{1+\frac{3\kappa}2} \right).
\end{multline}
In order to estimate the third term on the right hand side of \eqref{eq:mainBog}, we fix $1 < p < \overline{p}$ and $q > 1$ such that $\frac1\gamma+\frac1q+\frac1p=1$. Since the Bogovskii operator commutes with the derivative with respect to time, we deduce
\begin{multline*}
    \partial_t \boldsymbol\varphi_h = \phi_h \partial_t \mathcal{B}_\Omega\left[\rho^\alpha-\int_{\Omega^\eta(t)}\rho^\alpha\di \bb{y}\right] + \partial_t \phi_h \mathcal{B}_\Omega\left[\rho^\alpha-\int_{\Omega^\eta(t)}\rho^\alpha\di \bb{y}\right] \\
     = \phi_h \mathcal{B}_\Omega\left[\partial_t \left( \rho^\alpha-\int_{\Omega^\eta(t)}\rho^\alpha\di \bb{y}\right)\right] +\partial_t \phi_h \mathcal{B}_\Omega\left[\rho^\alpha-\int_{\Omega^\eta(t)}\rho^\alpha\di \bb{y}\right].
\end{multline*}
The continuity equation implies
\begin{equation*}
  \partial_t \rho^\alpha = - \nabla\cdot (\rho^\alpha \bb{u}) + (1 - \alpha)\rho^\alpha \nabla\cdot \bb{u}  
\end{equation*}
which is used to estimate
\begin{multline*}
 \left\|\mathcal{B}_\Omega\left[\partial_t \rho^\alpha-\partial_t\int_{\Omega^\eta(t)}\rho^\alpha\di \bb{y}\right]\right\|_{L^2(0,T;L^p(\Omega^\eta(t)))} \\
= \left\|\mathcal{B}_\Omega\left[ \nabla\cdot ( \rho^\alpha \bb{u}) + (\alpha-1)\rho^\alpha \nabla\cdot \bb{u} - (\alpha-1)\left(\int_{\Omega^\eta(t)}\rho^\alpha \nabla\cdot \bb{u}\di \bb{y} \right)\right]   \right\|_{L^2(0,T;L^p(\Omega^\eta(t)))} \\
\leq \|\rho^\alpha \bb{u}\|_{L^2(0,T;L^p(\Omega^\eta(t)))} + C \|\mathcal{B}_\Omega[\rho^\alpha \nabla\cdot \bb{u}] \|_{L^2(0,T;L^p(\Omega^\eta(t)))} \\
\leq \|\rho^\alpha \bb{u}\|_{L^2(0,T;L^p(\Omega^\eta(t)))} + C \|\rho^\alpha \nabla\cdot \bb{u} \|_{L^2(0,T;L^{r}(\Omega^\eta(t)))} \\
\leq \|\rho^\alpha\|_{L^\infty(0,T;L^\frac{\gamma}{\alpha}(\Omega^\eta(t)))} \| \bb{u}\|_{L^2(0,T;L^{\frac{p\gamma}{\gamma-\alpha p}}(\Omega^\eta(t)))} + C \|\rho^\alpha\|_{L^\infty(0,T;L^\frac{\gamma}{\alpha}(\Omega^\eta(t)))} \|\nabla \cdot\bb{u} \|_{L^2(Q_T^\eta)} \\
\leq C(\kappa)\left(1+\mathcal{E}^{\frac{\alpha}{\gamma}+\frac\kappa2}\right), 
\end{multline*}
where $r = \max\{1,\frac{2p}{2+p}\}$. Since
\begin{equation*}
    \partial_t \phi_h = -\frac{1}{h} \eta_t
\end{equation*}
on the set where it is not zero, it holds that
\begin{multline}\label{eq:Bogest13}
    \left|\int_{Q_T^\eta} \rho \bb{u}\cdot \partial_t \boldsymbol\varphi_h\di \bb{y} \dd t\right| \leq \|\rho\|_{L^\infty(0,T;L^\gamma(\Omega^\eta(t)))} \|\bb{u}\|_{L^2(0,T;L^q(\Omega^\eta(t)))} \| \partial_t \boldsymbol\varphi_h \|_{L^2(0,T;L^p(\Omega^\eta(t)))} \\ 
    \leq C(\kappa)\left(1+ \mathcal{E}^{\frac{1}{\gamma}+\frac\kappa2}\right) \left( \| \phi_h\|_{L^\infty(Q_T^\eta)} \mathcal{E}^{\frac{\alpha}{\gamma}+\frac{\kappa}{2}} + \| \eta_t\|_{L^2(0,T;L^p(\Gamma))} m_0^\alpha \right) \leq C(\kappa)\left(1+ \mathcal{E}^{\frac{1}{\gamma}+\frac{\alpha}{\gamma}+\kappa}\right).
\end{multline}
We continue with the fourth term on the right hand side of \eqref{eq:mainBog}. Here we take $q = \frac{2\gamma}{\gamma-1-\alpha}$ and deduce
\begin{multline*}
    \left|\int_{Q_T^\eta} \rho \bb{u}\otimes \bb{u} : \nabla \boldsymbol\varphi_h\di \bb{y} \dd t\right| \leq \|\rho\|_{L^\infty(0,T;L^\gamma(\Omega^\eta(t)))} \|\bb{u}\|_{L^2(0,T;L^q(\Omega^\eta(t)))}^2 \|\nabla \boldsymbol\varphi_h\|_{L^\infty(0,T;L^{\frac{\gamma}{\alpha}}(\Omega^\eta(t)))}\\
    \leq C(\kappa)\mathcal{E}^{\frac{1}{\gamma}+\kappa} \left( \left\|\nabla\mathcal{B}_\Omega\left[\rho^\alpha-\int_{\Omega}\rho^\alpha\di \bb{y}\right]\right\|_{L^\infty(0,T;L^{\frac{\gamma}{\alpha}}(\Omega^\eta(t)))} + \|\nabla\phi_h\|_{L^\infty(0,T;L^{\frac{\gamma}{\alpha}}(\Omega^\eta(t)))} m_0^\alpha  \right)\\
    \qquad \leq C(\kappa)\left(1+\mathcal{E}^{\frac{1}{\gamma}+\kappa})\big(  \|\rho^\alpha\|_{L^\infty(0,T;L^{\frac{\gamma}{\alpha}}(\Omega^\eta(t)))} + 1 + \|\eta_x \|_{L^\infty(0,T;L^{\frac{\gamma}{\alpha}}(\Gamma) )}  \right)\\
    \leq C(\kappa)(1+\mathcal{E}^{\frac{1}{\gamma}+\kappa})\Big( \|\rho^\alpha\|_{L^\infty(0,T;L^{\frac{\gamma}{\alpha}}(\Omega^\eta(t)))} + 1+\|\eta_{x}\|_{L^2(0,T; H^1(\Gamma))} + \|\eta_{tx}\|_{L^2(0,T; L^2(\Gamma))} \Big)\\
    \leq C(\kappa)(1+\mathcal{E}^{\frac{1}{\gamma}+\kappa})\Big(1+\mathcal{E}^{\frac{\alpha}{\gamma}}+ \mathcal{E}^{\frac{3\kappa}2} \Big)
    \leq C(\kappa)\left(1 + \mathcal{E}^{\frac{1+\alpha}\gamma+\kappa}+ \mathcal{E}^{\frac{1}\gamma +\frac{5\kappa}2}\right)
\end{multline*}
by \eqref{eq:int_est} and \eqref{eq:int_est_final}. The elliptic term satisfies
\begin{multline*}
    \left|\int_{Q_T^\eta} \mathbb{S}(\nabla \bb{u}): \nabla \boldsymbol\varphi_h\di \bb{y} \dd t\right| \leq \|\mathbb{S}(\nabla \bb{u})\|_{L^2(0,T;L^2(\Omega^\eta(t)))} \|\nabla \boldsymbol\varphi_h\|_{L^2(0,T;L^{\frac{\gamma}{\alpha}}(\Omega^\eta(t)))}\\
    \leq C(\kappa)\left(1+\mathcal{E}^{\frac\kappa2}\right)\left( \left\|\nabla\mathcal{B}_\Omega\left[\rho^\alpha-\int_{\Omega}\rho^\alpha\di \bb{y}\right]\right\|_{L^2(0,T;L^{\frac{\gamma}{\alpha}}(\Omega^\eta(t)))} + \|\nabla\phi_h\|_{L^2(0,T;L^{\frac{\gamma}{\alpha}}(\Omega^\eta(t)))} m_0^\alpha  \right)\\
    \leq C(\kappa)\left(1+\mathcal{E}^{\frac\kappa2}\right)\left(  \|\rho^\alpha\|_{L^\infty(0,T;L^{\frac{\gamma}{\alpha}}(\Omega^\eta(t)))} + (1+ \|\eta_x \|_{L^2(0,T;L^{\frac{\gamma}{\alpha}}(\Gamma) )})   \right)\\
    \leq C(\kappa)\left(1+\mathcal{E}^{\frac\kappa2}\right)\left( 1+ \|\rho^\alpha\|_{L^\infty(0,T;L^{\frac{\gamma}{\alpha}}(\Omega^\eta(t)))} + \|\eta_{xx} \|_{L^2(0,T;L^2(\Gamma) )}   \right)
    \leq C(\kappa)(1+\mathcal{E}^{\frac{\alpha}\gamma +\frac{\kappa}2} + \mathcal{E}^{2\kappa}).
\end{multline*}
Finally, 
\begin{equation*}
    \left|\int_{Q_T^\eta} \rho \bb{F} \cdot \boldsymbol\varphi_h\di \bb{y} \dd t\right| \leq C\|\rho\|_{L^\infty(0,T;L^1(\Omega^\eta(t)))} \|\bb{F}\|_{L^2(0,T;L^\infty(\Omega^\eta(t)))} \|\boldsymbol\varphi_h\|_{L^\infty(Q_T^\eta)} \leq Cm_0^{1+\alpha} \leq C.
\end{equation*}

We observe that due to \eqref{epcond} and \eqref{eq:alpha1} the largest power of $\mathcal{E}$ in all of the above estimates is $\mathcal{E}^{1+\frac{3\kappa}2}$. We combine these estimates to get
\begin{equation}\label{improved:int:est}
    \int_0^T \int_{\{\eta+h<z<\eta+2H-h\}} \rho^{\gamma+\alpha}\di \bb{y} \dd t\leq  \int_{Q_T^\eta} \rho^{\gamma+\alpha}\phi_h \di \bb{y} \dd t\leq C(\kappa)\left(1+\mathcal{E}^{1+\frac{3\kappa}2}\right),
\end{equation}
which then gives us by the interpolation of Lebesgue spaces
\begin{multline*}
    \left(\int_0^T \int_{\{\eta+h<z<\eta+2H-h\}} \rho^{\gamma}\di \bb{y} \dd t\right)^{\frac1\gamma}\\
    \leq \left( \int_0^T \int_{\{\eta+h<z<\eta+2H-h\}} \rho^{\gamma+\alpha}\di \bb{y} \dd t\right)^{\frac{\theta}{\gamma+\alpha}} \left( \int_0^T \int_{\{\eta+h<z<\eta+2H-h\}} \rho\di \bb{y} \dd t\right)^{1-\theta}\\
    \leq C(\kappa)(1+\mathcal{E}^{1+\frac{3\kappa}2})^{\frac{\theta}{\gamma+\alpha}}m_0^{1-\theta},
\end{multline*}
where
\begin{equation*}
    \theta=\frac{(\gamma-1)(\gamma+\alpha)}{(\gamma+\alpha-1)\gamma}.
\end{equation*}
The choice of $\kappa$ and $\alpha$ which satisfy \eqref{epcond} and \eqref{eq:alpha1} ensures that
\begin{equation}\label{eps:epsprime}
\left(1+\frac{3\kappa}2\right)\frac{\gamma\theta}{\gamma + \alpha} = \left(1+\frac{3\kappa}2\right)\frac{\gamma-1}{\gamma+\alpha-1} <1.
\end{equation}
We define
\begin{equation*}
    \kappa':= 1 - \left(1+\frac{3\kappa}2\right)\frac{\gamma-1}{\gamma+\alpha-1}
\end{equation*}
which yields 
\begin{equation}
    \int_0^T \int_{\{\eta+h<z<\eta+2H-h\}} \rho^{\gamma}\di \bb{y} \dd t \leq C(\kappa)\left(1+\mathcal{E}^{1-\kappa'}\right)\label{interest}.
\end{equation}
 
Next, we deal with the near boundary estimates. Recall that we have fixed $0<h<\frac{{H}}{2}$. This time we define
\begin{equation}\label{eq:varphi_h_def2}
    \varphi_h(t,x,z):=\begin{cases}
    z-\eta(t,x), &\quad \text{ for } \eta(t,x)<z<\eta(t,x)+h,\\
    -\frac{h}{H-h}(z-(\eta(t,x)+h))+h, &\quad \text{ for } \eta(t,x)+h<z<\eta(t,x)+2{H}-h,\\
    z-(\eta(t,x)+2H), &\quad \text{ for } \eta(t,x)+2{H}-h<z<\eta(t,x)+2{H}.
    \end{cases}
\end{equation}
Note that for fixed $(t,x)$, $\varphi_h(t,x,z)$ is piecewise linear in the $z$ variable with slope equal to $1$ near the boundary of the domain. We choose $(\boldsymbol\varphi,\psi)=(\varphi_h\bb{e}_2,0)$ as test functions in \eqref{momeqweak} to obtain 
\begin{multline}\label{eq:Bogbdry}
    \int_0^T \int_{\{\eta<z<\eta+h\}\cup\{\eta+2H-h<z<\eta+2H\}} \rho^{\gamma}\di \bb{y} \dd t\\ 
    = \frac{h}{H-h}\int_0^T \int_{\{\eta+h<z<\eta+2H-h\}} \rho^{\gamma}\di \bb{y} \dd t-\int_{Q_T^\eta} \rho \bb{u}\cdot \partial_t (\varphi_h\bb{e}_2)\di \bb{y} \dd t  \\ 
     - \int_{Q_T^\eta} \rho \bb{u}\otimes \bb{u} : \nabla (\varphi_h\bb{e}_2) \di \bb{y} \dd t + \int_{Q_T^\eta} \mathbb{S}(\nabla \bb{u}): \nabla (\varphi_h\bb{e}_2)\di \bb{y} \dd t - \int_{Q_T^\eta} \rho\bb{F}\cdot(\varphi_h\bb{e}_2)\di \bb{y} \dd t .
\end{multline}
We use \eqref{interest} to bound the first term on the right hand side. In order to bound the remaining terms, we use similar estimates as in the case of the interior estimates. In fact, the estimates are now more simple as there are no terms with the Bogovskii operator and the derivatives act directly on the function $\varphi_h$ and consequently on $\eta$. Therefore we obtain
\begin{multline}
    \int_0^T \int_{\{\eta<z<\eta+h\}\cup\{\eta+2H-h<z<\eta+2H\}} \rho^{\gamma}\di \bb{y} \dd t
    \leq C(\kappa)\left(1+\mathcal{E}^{1-\kappa'}\right)+C(\kappa)\left(1+ \mathcal{E}^{\frac{1}{\gamma}+\frac{\alpha}{\gamma}+\kappa}\right)\\
    +C(\kappa)\left(1 + \mathcal{E}^{\frac{1+\alpha}\gamma+\kappa}+ \mathcal{E}^{\frac{1}\gamma +\frac{5\kappa}2}\right)
+C(\kappa)\left(1+\mathcal{E}^{\frac{\alpha}\gamma +\frac{\kappa}2} + \mathcal{E}^{2\kappa}\right) 
    \leq C(\kappa)\left(1+\mathcal{E}^{1-\kappa''}\right),
    \label{bndest}
\end{multline}
where 
\begin{equation}\label{epprpr}
    \kappa'' := \min\left\{\kappa',1-\kappa-\frac{1+\alpha}{\gamma},1-\frac{1}{\gamma} - \frac{5\kappa}{2}\right\}.
\end{equation}
The conditions \eqref{epcond} and \eqref{eq:alpha1} ensure that $\kappa'' > 0$. We sum up \eqref{interest} and  \eqref{bndest} and we go back to $\Omega$ to finally deduce
\begin{equation*}
    \int_{Q_T} \rho^{\gamma}\di \bb{y} \dd t \leq C(\kappa)\left(1+\mathcal{E}^{1-\kappa''}\right),
\end{equation*}
where $\kappa$ and $\kappa''$ are related through \eqref{epprpr}.
\subsection{Part V - closing the estimates}\label{35sec}
We notice that for $q = \frac{2\gamma}{\gamma-1}$
\begin{equation*}
    \int_{Q_T} \rho |\bb{u}|^2\di \bb{y} \dd t\leq C \|\rho\|_{L^\infty(0,T;L^\gamma(\Omega))} \|\bb{u}\|_{L^2(0,T;L^q(\Omega))}^2 \leq C(\kappa)\left(1+\mathcal{E}^{\frac{1}{\gamma}+\kappa}\right).
\end{equation*}
Since $\frac{1}{\gamma} + \kappa < 1-\kappa''$ we finally obtain by previous estimates
\begin{equation*}
    \int_0^T E(s)\di s \leq C(\kappa)\left(1+\mathcal{E}^{1-\kappa''}\right) \leq C(\delta_0)+\delta_0 \mathcal{E}
\end{equation*}
for any $\delta_0 > 0$. 
This together with \eqref{eq:CIRC1} yields
\begin{equation*}
    \mathcal{E} \leq C_0\left(1 + \int_0^T E(s)\di s\right) \leq C_0(1 + \delta_0\mathcal{E} + C(\delta_0))
\end{equation*}
and, consequently, 
\begin{equation*}
    \mathcal{E} \leq C,
\end{equation*}
where $C$ depends on $f,\mathbf{F},m_0,L,H$, $h$ and the choice of $\kappa$. However, we can choose $h = \frac{H}{4}$, and the choice of $\kappa$ depends only on the value of $\gamma$ so the constant $C$ in the end depends only on $f,\mathbf{F},m_0,\gamma, L$ and $H$, i.e. the given data and parameters of the problem.

\section{Approximate decoupled problem}

We introduce the orthogonal basis of $L_\#^2(0,T)$ denoted by $\{\tau_i(t)\}_{i\in \mathbb{N}\cup \{0\}}$, more precisely we set for $k \in \mathbb{N}\cup \{0\}$
\begin{equation*}
    \tau_{2k}(t)=\cos\left(\frac{2\pi kt}{T}\right),\qquad \tau_{2k+1}(t)=\sin\left(\frac{2\pi kt}{T}\right).
\end{equation*}
We denote by $\{s_i(x)\}_{i\in \mathbb{N}}$ the orthogonal basis of $H^1_{\#,0}(\Gamma)\cap H_\#^2(\Gamma)$ and by $\{\bb{f}_i(x,z)\}_{i\in \mathbb{N}}$ the orthogonal basis of $H_\#^1(\Omega)$. We define finite-dimensional spaces
\begin{equation*}
\begin{split}
    \mathcal{P}_{n,m}^{str} &:= {\rm span}\{s_i(x)\tau_j(t)\}_{1\leq i\leq n, 0\leq j\leq 2m}, \\ \mathcal{P}_{n,m}^{fl} &:= {\rm span}\{\bb{f}_i(x,z)\tau_j(t)\}_{1\leq i\leq n,0\leq j\leq 2m}.
\end{split}
\end{equation*}
We fix $m,n \in \mathbb{N}$, we introduce parameters $\ep > 0$ and $\delta > 0$, and we fix $a \geq 5$. Here, $\ep$ denotes the artificial diffusion in the continuity equation, but also denotes the penalization parameter between the trace of the fluid velocity field on the viscoelastic beam and the velocity of the beam itself. The parameter $\delta$ then denotes an artificial pressure coefficient $\delta\rho^a$ in the momentum equation and it appears in other artificial terms which help us to get good estimates at the beginning of the proof but have to disappear from the equations later. 

We are ready to present the approximate decoupled and penalized problem which is the starting point of our existence proof. We fix $\beta \in (0,1)$, our goal is to find $\rho \in C_\#^{0,\beta}(0,T; C_\#^{2,\beta}(\Omega))\cap C_\#^{1,\beta}(0,T; C_\#^{0,\beta}(\Omega))$, $\bb{u} \in \mathcal{P}_{n,m}^{fl}$ and $\eta \in \mathcal{P}_{n,m}^{str}$ which satisfy the following identities.
\begin{enumerate}
    \item The \textbf{structure momentum equation}
    \begin{equation}\label{APP1}
         \int_{\Gamma_T} \eta_t  \psi_t\di x \dd t -    \int_{\Gamma_T} \eta_{xx}   \psi_{xx}\di x \dd t -\int_{\Gamma_T}  \eta_{tx}  \psi_x\di x \dd t -\int_{\Gamma_T}\frac{\eta_t  - \bb{v}\cdot\bb{e}_2}\varepsilon     \psi\di x \dd t=-\int_{\Gamma_T}f\psi\di x \dd t
    \end{equation}
holds for all $\psi \in \mathcal{P}_{n,m}^{str}$, where $\bb{v}=\gamma_{|\hat{\Gamma}^\eta}\bb{u}$.
\item The \textbf{damped continuity equation}
\begin{equation}\label{APP2}
    \partial_t \rho + \nabla \cdot (\rho \bb{u}) -\varepsilon\Delta \rho +\varepsilon \rho=\varepsilon M,
\end{equation}
complemented with periodic boundary conditions for $\rho$ holds in the classical sense in $\Omega$, where $M = \frac{m_0}{|\Omega|}$.
\item The \textbf{fluid momentum equation}
\begin{multline}
    \delta\int_{Q_T}\bb{u}\cdot\partial_t\boldsymbol\varphi \di \bb{y} \dd t + \int_{Q_T} \rho \bb{u} \cdot \partial_t \boldsymbol\varphi\di \bb{y} \dd t + \int_{Q_T} \rho \bb{u}\otimes \bb{u}:\nabla\boldsymbol\varphi\di \bb{y} \dd t +  \int_{Q_T} (\rho^\gamma+\delta\rho^a) \nabla\cdot\boldsymbol\varphi\di \bb{y} \dd t \\
    - \int_{Q_T} \mathbb{S}(\nabla \bb{u}):\nabla\boldsymbol\varphi\di \bb{y} \dd t
    -\delta \int_{Q_T} |\bb{u}|^2 \bb{u}\cdot\boldsymbol\varphi\di \bb{y} \dd t-\varepsilon\int_{Q_T}\nabla \rho\otimes\boldsymbol\varphi:\nabla \bb{u}\di \bb{y} \dd t\\
    + \frac\varepsilon2\int_{Q_T}(M-\rho)\bb{u}\cdot\boldsymbol\varphi\di \bb{y} \dd t
    -\int_{\Gamma_T}\frac{\bb{v}-\eta_t\bb{e}_2 }\varepsilon \cdot \boldsymbol\psi\di x \dd t=-\int_{Q_T}\rho\bb{F}_\delta\cdot\boldsymbol\varphi\di \bb{y} \dd t,
    \label{APP3}
\end{multline}
holds for all $\boldsymbol\varphi\in \mathcal{P}_{n,m}^{fl}$, where $\boldsymbol\psi=\gamma_{|\hat{\Gamma}^\eta}\boldsymbol\varphi$ and $\bb{v}=\gamma_{|\hat{\Gamma}^\eta}\bb{u}$. Here $\bb{F}_\delta$ denotes a smooth approximation of $\bb{F}$.
\end{enumerate}

\subsection{Uniform estimates}
We derive the uniform estimates for solutions to the approximate problem \eqref{APP1}-\eqref{APP3}. We choose $\psi=\eta_t$ in \eqref{APP1}, multiply \eqref{APP2} with $\frac{\gamma}{\gamma-1}\rho^{\gamma-1}$, then $\frac{\delta a}{a-1}\rho^{a-1}$ and $\frac{1}{2}|\bb{u}|^2$ and finally choose $\boldsymbol\varphi=\bb{u}$ in \eqref{APP3}, and then sum up these identities to obtain
\begin{multline}
     \int_{Q_T} \mathbb{S}(\nabla \bb{u}):\nabla \bb{u}\di \bb{y} \dd t+\delta \int_{Q_T}|\bb{u}|^4\di \bb{y} \dd t+\int_{\Gamma_T}|\eta_{tx} |^2\di x \dd t + \varepsilon\gamma\int_{Q_T} \rho^{\gamma-2}|\nabla \rho|^2\di \bb{y} \dd t 
    \\
    + \frac{\varepsilon\gamma}{\gamma-1} \int_{Q_T}\rho^{\gamma}\di \bb{y} \dd t 
     + \varepsilon\delta a\int_{Q_T} \rho^{a-2}|\nabla \rho|^2\di \bb{y} \dd t +\frac{\varepsilon\delta a}{a-1} \int_{Q_T}\rho^{a}\di \bb{y} \dd t
     + \frac1{\varepsilon}\int_{\Gamma_T} |\bb{v} - \eta_t\bb{e}_2|^2\di x \dd t \\
     = \int_{\Gamma_T} f\eta_t\di x \dd t  + \int_{Q_T}\rho\bb{u}\cdot\bb{F}_\delta\di \bb{y} \dd t+ \varepsilon\int_{Q_T} M\frac{\gamma}{\gamma-1}\rho^{\gamma-1}\di \bb{y} \dd t+\varepsilon\delta\int_{Q_T} M\frac{a}{a-1}\rho^{a-1}\di \bb{y} \dd t\\
      \leq \|f\|_{L^2(\Gamma_T)}\|\eta_{t}\|_{L^2(\Gamma_T)}+ C\|\rho\|_{L^{a}(Q_T)}\|\bb{u}\|_{L^4(Q_T)}\|\bb{F}_\delta\|_{L^\infty(Q_T)}\\
      + \frac{\varepsilon\gamma}{4(\gamma-1)} \|\rho\|_{L^{\gamma}(Q_T)}^{\gamma}+ \frac{\varepsilon\delta a}{4(a-1)}\|\rho\|_{L^{a}(Q_T)}^{a}+C(\varepsilon,\delta) \\
      \leq C(\varepsilon,\delta)+\frac12 \int_{\Gamma_T}|\eta_{tx} |^2\di x \dd t + \frac{\varepsilon\gamma}{4(\gamma-1)}\|\rho\|_{L^{\gamma}(Q_T)}^{\gamma}+\frac{\varepsilon\delta a}{2(a-1)}\|\rho\|_{L^{a}(Q_T)}^{a}+\frac\delta2\|\bb{u}\|_{L^4(Q_T)}^4,\label{appen1}
\end{multline}
where we used
\begin{multline*}
    \|\rho\|_{L^{a}(Q_T)}\|\bb{u}\|_{L^4(Q_T)}\|\bb{F}_\delta\|_{L^\infty(Q_T)} \leq C\|\rho\|_{L^{a}(Q_T)}\|\bb{u}\|_{L^4(Q_T)}\\
    \leq\frac{\varepsilon\delta a}{4(a-1)}\|\rho\|_{L^{a}(Q_T)}^{a}+\frac\delta2\|\bb{u}\|_{L^4(Q_T)}^4+ C(\varepsilon,\delta)
\end{multline*}
which follows from the Young inequality. Some terms on the right hand side of \eqref{appen1} might be absorbed in the left hand side and thus we deduce
\begin{multline*}
    \int_{Q_T} \mathbb{S}(\nabla \bb{u}):\nabla \bb{u}\di \bb{y} \dd t+\delta \int_{Q_T}|\bb{u}|^4\di \bb{y} \dd t+\int_{\Gamma_T}|\eta_{tx} |^2\di x \dd t+ \varepsilon\gamma\int_{Q_T} \rho^{\gamma-2}|\nabla \rho|^2\di \bb{y} \dd t+ \frac{\varepsilon\gamma}{\gamma-1} \int_{Q_T}\rho^{\gamma}\di \bb{y} \dd t\\
     + \varepsilon\delta a\int_{Q_T} \rho^{a-2}|\nabla \rho|^2\di \bb{y} \dd t+\frac{\varepsilon\delta a}{a-1} \int_{Q_T}\rho^{a}\di \bb{y} \dd t
      + \frac1{\varepsilon}\int_{\Gamma_T} |\bb{v} - \eta_t\bb{e}_2|^2 \di x \dd t \leq C(\varepsilon,\delta).
\end{multline*}
Next, we integrate \eqref{APP2} over $\Omega$ to deduce
\begin{equation*}
    \frac{d}{dt}\int_{\Omega} \rho(t)\di \bb{y} + \ep \int_{\Omega} \rho(t)\di \bb{y} = \ep m_0,  
\end{equation*}
which yields the only time-periodic solution
\begin{equation*}
    \int_{\Omega} \rho(t)\di \bb{y} = m_0.  
\end{equation*}


Further estimates of density are deduced by the $L^p - L^q$ theory for parabolic equations applied to the continuity equation \eqref{APP2}. To this end, we  estimate the term 
\begin{equation*}
\nabla \cdot (\rho \bb{u}) = \rho \nabla \cdot \bb{u} + \bb{u}\cdot \nabla \rho
\end{equation*}
in $L^p(0,T,L^q(\Omega))$ using the information we already have. The term $\rho\nabla\cdot\bb{u}$ is easy, as we have bounds for $\rho \in L^a(Q_T)$ and $\nabla\bb{u} \in L^2(Q_T)$. For the other term we use the bound $\bb{u} \in L^4(Q_T)$ and $\nabla\rho \in L^2(Q_T)$, where the latter follows from a straightforward manipulation with the continuity equation. Hence, we end up with
\begin{equation*}
     \|\partial_t \rho\|_{L^p(0,T;L^q(\Omega))}+\|\Delta \rho\|_{L^p(0,T;L^q(\Omega))} \leq C(\ep,\delta) 
\end{equation*}
for some $p,q\in (1,2)$, more specifically one can take $p=q=\frac43$.
Finally, we choose $\psi=\eta$ in \eqref{APP1} to obtain
\begin{equation*}
    \int_{\Gamma_T} |\eta_{xx}  |^2\di x \dd t = \frac1\varepsilon\int_{\Gamma_T} \bb{v}\cdot\bb{e}_2 \eta\di x \dd t+\int_{\Gamma_T} |\eta_t |^2\di x \dd t+ \int_{\Gamma_T} f\eta \di x \dd t\\
    \leq C(\varepsilon,\delta)+\frac12\int_{\Gamma_T} |\eta_{xx}  |^2\di x \dd t.
\end{equation*}
\noindent
To sum up, we have the following set of estimates independent of $m,n \in \mathbb{N}$.
\begin{equation}
\begin{split}\label{cor1}
    \| \eta_{tx} \|_{L^2(\Gamma_T)}&\leq C(\varepsilon,\delta), \\
     \|\bb{u}\|_{L^4(Q_T)}&\leq C(\varepsilon,\delta), \\
    \|\bb{u}\|_{L^2(0,T;H^1(\Omega))}&\leq C(\varepsilon,\delta), \\
    \|\bb{u} \|_{L^2(0,T;L^p(\Omega))}&\leq C(\varepsilon,\delta,p), \quad \text{for any } p\in(1,\infty), \\
    \|\rho\|_{L^a(Q_T)}&\leq C(\varepsilon,\delta), \\
   \|\partial_t \rho\|_{L^p(0,T;L^q(\Omega))}+\|\Delta \rho\|_{L^p(0,T;L^q(\Omega))}&\leq C(\varepsilon,\delta,p,q), \quad \text{for some } p,q\in(1,2),\\
   \|\eta\|_{L^2(0,T;H^2(\Gamma))}& \leq C(\varepsilon,\delta).
\end{split}    
\end{equation}
\subsection{Solution to the approximate problem}
\begin{lem}\label{SSPlem}
Assume  $f\in L^2_\#(\Gamma_T)$, $\tilde{\bb{u}}\in \mathcal{P}_{n,m}^{fl}$, and $\tilde{\eta}\in  \mathcal{P}_{n,m}^{str}$ are given and let $\tilde{\bb{v}}=\gamma_{|\hat{\Gamma}^{\tilde{\eta}}}\tilde{\bb{u}}$ (or equivalently $\tilde{\bb{v}}(t,x)=\tilde{\bb{u}}(t,x,\tilde{\eta}(t,x))$). Then, the following problem 
\begin{equation}\label{appstruct1b}
         \int_{\Gamma_T} \eta_{tt} \psi \di x \dd t +    \int_{\Gamma_T} \eta_{xx}   \psi_{xx}\di x \dd t +\int_{\Gamma_T} \eta_{tx}  \psi_x\di x \dd t +\int_{\Gamma_T}\frac{\eta_t  - \tilde{\bb{v}}\cdot\bb{e}_2}\varepsilon     \psi\di x \dd t=\int_{\Gamma_T}f\psi\di x \dd t
 \end{equation}
 for all $\psi\in\mathcal{P}_{n,m}^{str}$ and all $t \in (0,T)$ has a unique solution $\eta\in \mathcal{P}_{n,m}^{str}$. Moreover, the mapping $(\tilde{\bb{u}},\tilde{\eta})\mapsto \eta$ is compact from $\mathcal{P}_{n,m}^{fl}\times \mathcal{P}_{n,m}^{str}$ to $\mathcal{P}_{n,m}^{str}$.
\end{lem}

\begin{proof}
The idea is to solve $\eqref{appstruct1b}$ in $\eta_t$ instead of $\eta$. Note that, due to time periodicity of $\eta$, function $\eta_t$ must be mean-value free in time and therefore cannot contain the constant function in time from the time basis. Therefore, we define $S_0 = \mathcal{P}^{str}_{n,0} = {\rm span}\{s_i(x)\}_{1\leq i\leq n}$ and $S:=(\mathcal{P}_{n,m}^{str}\setminus S_0, ||\cdot||_{L^2(\Gamma_T)})$ and the mappings $B:S\times S\to \mathbb{R}$ and $a:S\to \mathbb{R}$ as
\begin{eqnarray*}
    &&B(u,v):= \int_{\Gamma_T} u_{t} v \di x\dd t+ \int_{\Gamma_T} U_{xx} v_{xx}\di x\dd t+ \int_{\Gamma_T} u_{x} v_x \di x\dd t+ \int_{\Gamma_T}\frac{u}{\varepsilon}v\di x\dd t,\\
    &&a(v) = \int_{\Gamma_T} f v\di x\dd t+ \int_{\Gamma_T}\frac{ \tilde{\bb{v}}\cdot\bb{e}_2}\varepsilon v\di x\dd t
\end{eqnarray*}
where $U(t,x):=\int_0^t u(s,x) \di s$. Then, our problem can be formulated as finding $\eta_t=u \in S$ such that $B(u,v)=a(v)$ for all $v\in S$. Obviously, $B$ is bi-linear and $a$ is bounded and linear. Moreover, by the equivalence of norms in finite basis $\mathcal{P}_{n,m}^{str}$, one has $B(u,v)\leq C ||u||_{L^2(\Gamma_T)} ||v||_{L^2(\Gamma_T)}$. Finally, due to time-periodicity, one has
\begin{eqnarray*}
    B(u,u) = ||u_x||_{L^2(\Gamma_T)}^2 + \frac1\varepsilon ||u||_{L^2(\Gamma_T)}^2 \geq C ||u||_{L^2(\Gamma_T)}^2.
\end{eqnarray*}
Therefore, the solution $\eta_t=u\in S$ follows directly by Lax-Milgram Lemma. Since $\int_0^t \eta_t(s,x)\,{\rm d} s$ in general does not belong to the space $S$ due to integrals of $\tau_{2k+1}(t)$, we find $\eta$ in the form $\eta(t,x) = P_S(\int_0^t \eta_t(s,x)ds) + G(x)$, where $P_S$ is a projection from $\mathcal{P}^{str}_{n,m}$ onto the space $S$ and $G(x)\in S_0$ is a solution to 
the elliptic equation
\begin{eqnarray*}
    - \int_{\Gamma_T} G_{xx} \psi_{xx}\, {\rm d}x   + \int_{\Gamma_T}  \frac{\tilde{\bb{v}}\cdot\bb{e}_2}\varepsilon     \psi\, {\rm d}x=-\int_{\Gamma_T}f\psi\, {\rm d}x
\end{eqnarray*}
for all $\psi \in S_0$. 
The continuity of mapping $(\tilde{\bb{u}},\tilde{\eta})\mapsto \eta$ is a direct consequence of linearity of the equation.
\end{proof}

\begin{lem}\label{lem1}(\cite[Lemma 2]{FMNP})
Let $\tilde{\bb{u}}\in \mathcal{P}_{n,m}^{fl}$. Then, there exists a unique solution $\rho$ to the following problem
\begin{equation*}
    \partial_t \rho + \nabla \cdot (\rho \tilde{\bb{u}}) -\varepsilon\Delta \rho +\varepsilon \rho=\varepsilon M.
\end{equation*}
Moreover, $\rho\in C_\#^\infty(0,T;W_\#^{2,p}(\Omega))$ for any $p\in (1,\infty)$, the mapping $\tilde{\bb{u}}\mapsto \rho$ is continuous and compact from $\mathcal{P}_{n,m}^{fl}$ to $W_\#^{1,p}(Q_T)$ and $\rho\geq 0$.
\end{lem}

\begin{lem}\label{lem2}
Let $\tilde{\bb{u}}\in \mathcal{P}_{n,m}^{fl}$, $ \tilde{\eta}\in \mathcal{P}_{n,m}^{str}$ and $\rho\in C_\#^\infty(0,T;W_\#^{2,p}(\Omega))$. Then, there exists a solution $\bb{u}\in \mathcal{P}_{n,m}^{fl}$ of
\begin{multline}
    \delta\int_{Q_T} \bb{u} \cdot \partial_t \boldsymbol\varphi\di \bb{y} \dd t+\int_{Q_T} \rho \tilde{\bb{u}} \cdot \partial_t \boldsymbol\varphi\di \bb{y} \dd t + \int_{Q_T} \rho \tilde{\bb{u}}\otimes \tilde{\bb{u}}:\nabla\boldsymbol\varphi\di \bb{y} \dd t +  \int_{Q_T} (\rho^\gamma+\delta\rho^a) \nabla\cdot\boldsymbol\varphi\di \bb{y} \dd t\\
    - \int_{Q_T} \mathbb{S}(\nabla \bb{u}):\nabla\boldsymbol\varphi\di \bb{y} \dd t  - \delta \int_{Q_T}|\bb{u}|^2 \bb{u}\cdot\boldsymbol\varphi\di \bb{y} \dd t -\varepsilon\int_{Q_T}\nabla \rho\otimes\boldsymbol\varphi:\nabla \tilde{\bb{u}}\di \bb{y} \dd t\\
    + \frac\varepsilon2\int_{Q_T}(M-\rho)\tilde{\bb{u}}\cdot\boldsymbol\varphi\di \bb{y} \dd t
    -\int_{\Gamma_T}\frac{\bb{v}- \tilde{\eta}_t\bb{e}_2}\varepsilon \cdot\boldsymbol\psi\di x \dd t=-\int_{Q_T}\rho\bb{F}_\delta\cdot\boldsymbol\varphi\di \bb{y} \dd t,\label{eq:flap}
\end{multline}
for all $\boldsymbol\varphi\in \mathcal{P}_{n,m}^{fl}$, where $\boldsymbol\psi = \gamma_{|\hat{\Gamma}^{\tilde{\eta}}}\boldsymbol\varphi$ and $\bb{v} = \gamma_{|\hat{\Gamma}^{\tilde{\eta}}}\bb{u}$. Moreover, the mapping $(\rho,\tilde{\bb{u}},\tilde{\eta})\mapsto \bb{u}$ is continuous from $W^{1,p}_\#(Q_T)\times\mathcal{P}_{n,m}^{fl}\times \mathcal{P}_{n,m}^{str}$ to $\mathcal{P}_{n,m}^{fl}$.
\end{lem}
\begin{proof}
The existence of solution is straightforward. Indeed, \eqref{eq:flap} may be rewritten as
$$
A \bb{u} = RHS
$$
where 
$$
A\bb{u} =\mathcal P\left( \delta \bb{u}_t - \nabla\cdot\mathbb S(\nabla \bb{u}) + \delta |\bb{u}|^2 \bb{u} + \frac 1\varepsilon \bb{v}\right)
$$
where $\mathcal P$ denotes the projection to $\mathcal P^{fl}_{n,m}$ and $RHS$ contains all the other terms. The operator $A$ is a coercive operator on $\mathcal P_{n,m}^{fl}$ and the classical result then yields that $A$ is also surjective -- we refer to \cite[Theorem 2.6]{rouba}.

To prove the continuity, let $\rho_1,\rho_2\in C_\#^\infty(0,T;W_\#^{2,p}(\Omega))$, $\tilde{\bb{u}}_1,\tilde{\bb{u}}_2\in \mathcal{P}_{n,m}^{fl}$ and $\tilde{\eta}_1,\tilde{\eta_2}\in \mathcal{P}_{n,m}^{str}$ be given, and let $\bb{u}_1, \bb{u}_2\in \mathcal{P}_{n,m}^{fl}$ be the corresponding solutions. Denote $\bb{v}_i = \gamma_{|\hat{\Gamma}^{\tilde{\eta_i}}}\bb{u}_i$ for $i=1,2$. We take the difference of the equation for $\bb{u}_1$ tested with $\boldsymbol\varphi = (\bb{u}_1-\bb{u}_2)$ and the equation for $\bb{u}_2$ tested with $\boldsymbol\varphi = (\bb{u}_1-\bb{u}_2)$. We emphasize that even though the test functions $\boldsymbol\varphi$ in both equations are the same, the corresponding $\boldsymbol\psi$ are different in both equations, as they are traces of $\boldsymbol\varphi$ on different curves $\tilde{\eta}_i$. Since
\begin{equation*}
    \frac14 |\bb{u}_1-\bb{u}_2|^4 \leq (|\bb{u}_1|^2\bb{u}_1 - |\bb{u}_2|^2\bb{u}_2)\cdot(\bb{u}_1-\bb{u}_2)
\end{equation*}
we get
\begin{multline*}
    \int_{Q_T} \mathbb{S}(\nabla\bb{u}_1 - \nabla\bb{u}_2):\nabla(\bb{u}_1 - \bb{u}_2)\di \bb{y} \dd t+\frac{\delta}{4} \int_{Q_T}|\bb{u}_1-\bb{u}_2|^4\di \bb{y} \dd t+\frac1\varepsilon\int_{\Gamma_T} |\bb{v}_1-\bb{v}_2|^2\di x \dd t\\
    \leq\int_{Q_T} (\rho_1 \tilde{\bb{u}}_1 - \rho_2\tilde{\bb{u}}_2) \cdot \partial_t(\bb{u}_1 - \bb{u}_2)\di \bb{y} \dd t+\int_{Q_T} (\rho_1 \tilde{\bb{u}}_1\otimes \tilde{\bb{u}}_1 - \rho_2 \tilde{\bb{u}}_2\otimes \tilde{\bb{u}}_2): \nabla (\bb{u}_1-\bb{u}_2)\di \bb{y} \dd t \\
    +  \int_{Q_T} (\rho_1^\gamma -\rho_2^\gamma  +\delta\rho_1^a-\delta\rho_2^a) \nabla\cdot(\bb{u}_1-\bb{u}_2)\di \bb{y} \dd t- \varepsilon\int_{Q_T}(\nabla \rho_1- \nabla\rho_2)\otimes(\bb{u}_1 - \bb{u}_2):\nabla \tilde{\bb{u}}_1\di \bb{y} \dd t \\
    +\varepsilon\int_{Q_T}\nabla \rho_2\otimes(\bb{u}_1 - \bb{u}_2):\nabla (\tilde{\bb{u}}_2-\tilde{\bb{u}}_1)\di \bb{y} \dd t+ \frac\varepsilon2\int_{Q_T}M(\tilde{\bb{u}}_1 -\tilde{\bb{u}}_2) \cdot(\bb{u}_1-\bb{u}_2) \di \bb{y} \dd t\\
    - \frac\varepsilon2\int_{Q_T}(\rho_1 - \rho_2)\tilde{\bb{u}}_2\cdot(\bb{u}_1-\bb{u}_2)\di \bb{y} \dd t - \frac\varepsilon2\int_{Q_T} \rho_1(\tilde{\bb{u}}_1-\tilde{\bb{u}}_2)\cdot(\bb{u}_1-\bb{u}_2)\di \bb{y} \dd t\\
    +\int_{Q_T}(\rho_1 - \rho_2)\bb{F}_\delta \cdot (\bb{u}_1-\bb{u}_2)\di \bb{y} \dd t 
    -\frac1\varepsilon \int_{\Gamma_T} (\gamma_{|\hat{\Gamma}^{\tilde{\eta_1}}}\bb{u}_1 - \gamma_{|\hat{\Gamma}^{\tilde{\eta_2}}}\bb{u}_2) \cdot (\gamma_{|\hat{\Gamma}^{\tilde{\eta_2}}} \bb{u}_2 - \gamma_{|\hat{\Gamma}^{\tilde{\eta_1}}} \bb{u}_2) \di x \dd t\\
    - \frac1\varepsilon\int_{\Gamma_T} \gamma_{|\hat{\Gamma}^{\tilde{\eta_2}}}\bb{u}_2 \cdot (\gamma_{|\hat{\Gamma}^{\tilde{\eta_2}}}(\bb{u}_2-\bb{u}_1) - \gamma_{|\hat{\Gamma}^{\tilde{\eta_1}}}(\bb{u}_2-\bb{u}_1)) \di x \dd t\\
    +\frac1\varepsilon\int_{\Gamma_T}(\tilde{\eta}_{1t}- \tilde{\eta}_{2t})\bb{e}_2 \cdot \gamma_{|\hat{\Gamma}^{\tilde{\eta_1}}}(\bb{u}_1 - \bb{u}_2)\di x \dd t\\
    +\frac1\varepsilon\int_{\Gamma_T}\tilde{\eta}_{2t}\bb{e}_2 \cdot (\gamma_{|\hat{\Gamma}^{\tilde{\eta_1}}}(\bb{u}_1 - \bb{u}_2)-\gamma_{|\hat{\Gamma}^{\tilde{\eta_2}}}(\bb{u}_1 - \bb{u}_2))\di x \dd t
\end{multline*}
where we used that
\begin{multline*}
    \int_{\Gamma_T} \gamma_{|\hat{\Gamma}^{\tilde{\eta_1}}}\bb{u}_1 \cdot \gamma_{|\hat{\Gamma}^{\tilde{\eta_1}}}(\bb{u}_1 -\bb{u}_2)\di x \dd t -  \int_{\Gamma_T} \gamma_{|\hat{\Gamma}^{\tilde{\eta_2}}}\bb{u}_2 \cdot \gamma_{|\hat{\Gamma}^{\tilde{\eta_2}}}(\bb{u}_1 -\bb{u}_2)\di x \dd t\\
    = \int_{\Gamma_T} \gamma_{|\hat{\Gamma}^{\tilde{\eta_1}}}\bb{u}_1 \cdot (\gamma_{|\hat{\Gamma}^{\tilde{\eta_1}}}\bb{u}_1 -\gamma_{|\hat{\Gamma}^{\tilde{\eta_2}}}\bb{u}_2)\di x \dd t -  \int_{\Gamma_T} \gamma_{|\hat{\Gamma}^{\tilde{\eta_2}}}\bb{u}_2 \cdot( \gamma_{|\hat{\Gamma}^{\tilde{\eta_1}}}\bb{u}_1 - \gamma_{|\hat{\Gamma}^{\tilde{\eta_2}}}\bb{u}_2)\di x \dd t \\
    +\int_{\Gamma_T} \gamma_{|\hat{\Gamma}^{\tilde{\eta_1}}}\bb{u}_1 \cdot (\gamma_{|\hat{\Gamma}^{\tilde{\eta_2}}} \bb{u}_2 - \gamma_{|\hat{\Gamma}^{\tilde{\eta_1}}} \bb{u}_2)\di x \dd t - \int_{\Gamma_T} \gamma_{|\hat{\Gamma}^{\tilde{\eta_2}}}\bb{u}_2 \cdot (\gamma_{|\hat{\Gamma}^{\tilde{\eta_2}}} \bb{u}_1 - \gamma_{|\hat{\Gamma}^{\tilde{\eta_1}}} \bb{u}_1)\di x \dd t \\
     = \int_{\Gamma_T}  \underbrace{|\gamma_{|\hat{\Gamma}^{\tilde{\eta_1}}}\bb{u}_1 -\gamma_{|\hat{\Gamma}^{\tilde{\eta_2}}}\bb{u}_2|^2}_{=|\bb{v}_1-\bb{v}_2|^2}\di x \dd t +  \int_{\Gamma_T} (\gamma_{|\hat{\Gamma}^{\tilde{\eta_1}}}\bb{u}_1 - \gamma_{|\hat{\Gamma}^{\tilde{\eta_2}}}\bb{u}_2) \cdot (\gamma_{|\hat{\Gamma}^{\tilde{\eta_2}}} \bb{u}_2 - \gamma_{|\hat{\Gamma}^{\tilde{\eta_1}}} \bb{u}_2)\di x \dd t \\
     + \int_{\Gamma_T} \gamma_{|\hat{\Gamma}^{\tilde{\eta_2}}}\bb{u}_2 \cdot (\gamma_{|\hat{\Gamma}^{\tilde{\eta_2}}}(\bb{u}_2-\bb{u}_1) - \gamma_{|\hat{\Gamma}^{\tilde{\eta_1}}}(\bb{u}_2-\bb{u}_1))\di x \dd t.
\end{multline*}
The convective term is treated as follows
\begin{multline*}
    \int_{Q_T} (\rho_1 \tilde{\bb{u}}_1\otimes \tilde{\bb{u}}_1 - \rho_2 \tilde{\bb{u}}_2\otimes \tilde{\bb{u}}_2): \nabla (\bb{u}_1-\bb{u}_2)\di \bb{y} \dd t \\
    \leq C \int_{Q_T} (|\rho_1 - \rho_2|^2 + |\tilde{\bb{u}}_1-\tilde{\bb{u}}_2|^2)\di \bb{y} \dd t + c\int_{Q_T}|\nabla(\bb{u}_1-\bb{u}_2)|^2\di \bb{y} \dd t
\end{multline*}
where $c$ is taken small enough to absorb the term into the left hand side using the Korn inequality and $C$ depends on the functions and $m,n,\varepsilon,\delta$ and $c$.   The remaining terms on $Q_T$ are estimated in a similar fashion. The most involved boundary term is the following
\begin{multline*}
    \frac1\varepsilon\int_{\Gamma_T} \gamma_{|\hat{\Gamma}^{\eta_2}}\bb{u}_2 \cdot (\gamma_{|\hat{\Gamma}^{\eta_2}}(\bb{u}_2-\bb{u}_1) - \gamma_{|\hat{\Gamma}^{\eta_1}}(\bb{u}_2-\bb{u}_1))\di x \dd t \leq C \int_{\Gamma_T} ( \tilde{\eta}_1 - \tilde{\eta}_2) \|\partial_{z}\bb{u}_1 - \partial_{z}\bb{u}_2\|_{C(Q_T)}\di x \dd t\\
    \leq C\int_{\Gamma_T} |\tilde{\eta}_1 - \tilde{\eta}_2|^2\di x \dd t + \frac{\delta}{16} \int_{Q_T} |\bb{u}_1 - \bb{u}_2|^2\di \bb{y} \dd t,
\end{multline*}
by the equivalence of norms in a finite basis, where we have also used
\begin{multline*}
    \gamma_{|\hat{\Gamma}^{\eta_2}}(\bb{u}_2-\bb{u}_1)(t,x) - \gamma_{|\hat{\Gamma}^{\eta_1}}(\bb{u}_2-\bb{u}_1)(t,x) = (\bb{u}_2-\bb{u}_1)(t,x,\tilde{\eta}_1(t,x)) -(\bb{u}_2-\bb{u}_1)(t,x,\tilde{\eta}_2(t,x))\\
    =(\tilde{\eta}_1(t,x) - \tilde{\eta}_2(t,x)) \partial_{z}(\bb{u}_2-\bb{u}_1)(t,x,\theta\tilde{\eta}_1(t,x)+(1-\theta)\tilde{\eta}_2(t,x) ),\quad \theta \in (0,1),
\end{multline*}
which follows by the mean value theorem. We estimate the other terms similarly and we end up with
\begin{multline*}
     \int_{Q_T} \mathbb{S}(\nabla\bb{u}_1 - \nabla\bb{u}_2):\nabla(\bb{u}_1 - \bb{u}_2)\di \bb{y} \dd t+\delta \int_{Q_T}|\bb{u}_1-\bb{u}_2|^4\di \bb{y} \dd t+\frac1\varepsilon\int_{\Gamma_T} |\bb{v}_1-\bb{v}_2|^2\di x \dd t\\
    \leq C\int_{Q_T}|\nabla\rho_1 -\nabla\rho_2|^2\di \bb{y} \dd t + C\int_{Q_T}|\rho_1 -\rho_2|^2\di \bb{y} \dd t + C\int_{Q_T}|\tilde{\bb{u}}_1 -\tilde{\bb{u}}_2|^2\di \bb{y} \dd t \\
    +  C\int_{\Gamma_T} |\tilde{\eta}_1 - \tilde{\eta}_2|^2\di x \dd t +  C\int_{\Gamma_T} |\tilde{\eta}_{1t} - \tilde{\eta}_{2t}|^2\di x \dd t,
\end{multline*}
so the solution mapping is continuous.
\end{proof}
\begin{lem}
There exists a solution $(\rho,\bb{u},\eta)$ to the approximate problem $\eqref{APP1}-\eqref{APP3}$.
\end{lem}
\begin{proof}
We define an operator
\begin{equation*}
    \mathcal{T}:\begin{array}{l}\mathcal{P}_{n,m}^{str}\times \mathcal{P}_{n,m}^{fl} \to \mathcal{P}_{n,m}^{str}\times \mathcal{P}_{n,m}^{fl},\\
    (\tilde{\eta},\tilde{\bb{u}}) \mapsto (\eta,\bb{u}),
    \end{array}
\end{equation*}
where $\eta =\eta(\tilde{\bb{u}},\tilde{\eta})$ is obtained in Lemma \ref{SSPlem}, $\rho =\rho(\tilde{\bb{u}})$ is obtained in Lemma \ref{lem1} and $\bb{u}=\bb{u}(\rho,\tilde{\bb{u}},\tilde{\eta})$ is the solution obtained in Lemma \ref{lem2}. As a consequence of these lemmas, mapping $\mathcal{T}$ is continuous and it is compact.

It remains to show that the set 
\begin{equation}\label{schaefferset} 
    \{(\tilde\eta,\tilde{\bb{u}})\in \mathcal{P}_{n,m}^{str}\times \mathcal{P}_{n,m}^{fl}: \lambda \mathcal{T}(\tilde\eta,\tilde{\bb{u}})=(\tilde\eta,\tilde{\bb{u}}), \lambda \in [0,1] \}
\end{equation}
is bounded. We denote $(\eta,\bb{u})= \mathcal{T}(\tilde\eta,\tilde{\bb{u}})$ and emphasize that
points from \eqref{schaefferset} satisfy $\lambda(\eta,\bb{u}) = (\tilde \eta,\tilde{\bb{u}})$. We test \eqref{eq:flap} by $\boldsymbol{\varphi} = \tilde{\bb{u}} = \lambda\bb{u}$ and \eqref{appstruct1b} by $\psi = \eta_t$. Recalling $\rho = \rho(\tilde {\bf u})$ and making similar calculations as in \eqref{appen1} we obtain
\begin{multline*}
\lambda\int_{Q_T} \mathbb S(\nabla {\bf u}):\nabla {\bf u}\di \bb{y} \dd t + \lambda\delta \int_{Q_T} |{\bf u}|^4\di \bb{y} \dd t + \int_{\Gamma_T} |\eta_{tx}|^2 \di x \dd t   \\
+ \varepsilon \int_{Q_T} \rho^{\gamma-2}|\nabla \rho|^2\di \bb{y} \dd t + \frac{\varepsilon\gamma}{\gamma-1}  \int_{Q_T} \rho^\gamma \di \bb{y} \dd t + \varepsilon \delta a\int_{Q_T} \rho^{a-2}|\nabla \rho|^2\di \bb{y} \dd t + \frac{\varepsilon \delta a}{a-1}\int_{Q_T}\rho^a\di \bb{y} \dd t  \\
+ \frac{\lambda}{2\ep}\int_{\Gamma_T} |\bb{v}-\lambda\eta_t|^2\di x \dd t + \frac{\lambda}{2\ep}\int_{\Gamma_T} |\bb{v}|^2 \di x \dd t + \frac{1}{2\ep}\int_{\Gamma_T} |\lambda\bb{v}\cdot \bb{e}_2-\eta_t|^2\di x \dd t + \frac{1}{2\ep}\int_{\Gamma_T} |\eta_t|^2 \di x \dd t \\
= \int_{\Gamma_T} f \eta_t\di x \dd t + \lambda\int_{Q_T} \rho {\bf F}_\delta \cdot{\bf u}\di \bb{y} \dd t + \frac {\varepsilon M \gamma}{\gamma-1}  \int_{Q_T} \rho^{\gamma-1}\di \bb{y} \dd t  + \frac{\varepsilon \delta M a}{a-1} \int_{Q_T} \rho^{a-1}\di \bb{y} \dd t\\
+ \frac {\lambda^3}{2 \varepsilon} \int_{\Gamma_T} |\eta_t|^2\di x \dd t + \frac {\lambda^2}{2 \varepsilon}\int_{\Gamma_T} |\bb{v}\cdot\bb{e}_2|^2\di x \dd t
\end{multline*}
where $\bb{v} = \gamma_{|\hat{\Gamma}^{\tilde{\eta}}}\bb{u}$. The first four terms on the right hand side can be dealt with as in \eqref{appen1}. The last two terms can be easily absorbed to the left hand side as $\lambda \leq 1$. We obtain
\begin{equation*}
    \lambda\int_{Q_T} \mathbb S(\nabla {\bf u}):\nabla {\bf u}\di \bb{y} \dd t + \lambda\delta \int_{Q_T} |{\bf u}|^4\di \bb{y} \dd t + \int_{\Gamma_T} |\eta_{tx}|^2 \di x \dd t \leq C
\end{equation*}
which provides by multiplying with suitable powers of $\lambda$
\begin{equation*}
    \int_{Q_T} \mathbb S(\nabla \tilde{\bb{u}}):\tilde{\bb{u}}\di \bb{y} \dd t + \delta \int_{Q_T} |\tilde{\bb{u}}|^4\di \bb{y} \dd t + \int_{\Gamma_T} |\tilde{\eta}_{tx}|^2 \di x \dd t \leq C, 
\end{equation*}
hence the set \eqref{schaefferset} is bounded. The desired claim then follows by the Schaeffer fixed point theorem.

Finally, since $\rho$ is a solution to \eqref{APP2}, classical theory of parabolic equations implies H\" older regularity of $\rho$.

\end{proof}

\section{Time basis limit $m\to \infty$}
Denote the approximate solution obtained in previous section as $(\rho_{m},\bb{u}_m,\eta_m)$. One obtains from \eqref{APP3} and \eqref{cor1} that $\partial_t \bb{u}_m$ is bounded by a constant independent from $m$ in $L^1(0,T; \text{span}\{\bb{f}_i\}_{1\leq i\leq n})$. This means that $\bb{u}_m$ is bounded in $L^\infty(0,T; \text{span}\{\bb{f}_i\}_{1\leq i\leq n})$, so one can again estimate $\partial_t \bb{u}_m$ in a better space $L_\#^p(0,T; \text{span}\{\bb{f}_i\}_{1\leq i\leq n})$, for any $p<\infty$. Similarly, the equation $\eqref{APP1}$ implies $\partial_{tt}\eta_m\in L^p_\#(0,T; \text{span}\{s_i\}_{1\leq i\leq n})$ for any $p<\infty$. 
This together with \eqref{cor1} allow us to pass to the limit $m\to \infty$ in most terms in the system \eqref{APP1}-\eqref{APP3}. The following lemma allows us to pass to the limit in the trace terms.
\begin{lem}\label{tracelemma}
Let $\bb{u}_m \weak \bb{u}$ weakly in $L^2_\#(0,T;H^1_\#(\Omega))$ and let $\eta_m \weak \eta$ weakly in $L^\infty_\#(0,T;H^2_{\#}(\Gamma))$ and in $H^1_\#(0,T;H^1_{\#,0}(\Gamma))$. Then
\begin{equation*}
    \int_{\Gamma_T} \bb{u}_m(t,x,\eta_m(t,x)) \cdot \boldsymbol\psi(t,x)\di x \dd t \to \int_{\Gamma_T} \bb{u}(t,x,\eta(t,x)) \cdot \boldsymbol\psi(t,x)\di x \dd t
\end{equation*}
for all $\boldsymbol\psi \in C^\infty_{\#,0}(\Gamma_T)$.
\end{lem}
\begin{proof}
Denote $\tilde{\bb{u}}_m(t,x,z) = \bb{u}_m(t,x,z+\eta_m(t,x))$. The Sobolev embedding theorem implies $(\eta_m)_x$ is bounded in $L^\infty(\Gamma_t)$ and therefore $\tilde{\bb{u}}_m$ is bounded in $L^2_\#(0,T;H^1_\#(\Omega))$. We extract a subsequence converging to some $\bb{U}$ weakly in $L^2_\#(0,T;H^1_\#(\Omega))$. Our aim is to identify the limit as $\bb{U}(t,x,z) = \tilde{\bb{u}}(t,x,z) := \bb{u}(t,x,z+\eta(t,x))$. Denote $\bb{w}_m := \bb{u}_m - \bb{u}$. We have
\begin{equation*}
    (\tilde{\bb{u}}_m - \tilde{\bb{u}})(t,x,z) = \bb{w}_m(t,x,z+\eta_m(t,x)) + \bb{u}(t,x,z+\eta_m(t,x)) - \bb{u}(t,x,z+\eta(t,x))
\end{equation*}
Fix $\boldsymbol\varphi \in C^\infty_\#(Q_T)$. Then
\begin{equation*}
    \int_{Q_T} \bb{w}_m(t,x,z+\eta_m(t,x)) \cdot \boldsymbol\varphi(t,x,z)\di \bb{y} \dd t = \int_{Q_T} \bb{w}_m(t,x,z) \cdot \boldsymbol\varphi(t,x,z-\eta_m(t,x))\di \bb{y} \dd t,
\end{equation*}
where $\bb{w}_m$ converges weakly in $L^2_\#(0,T;H^1_\#(\Omega))$ to zero and $\boldsymbol\varphi(t,x,z-\eta_m(t,x))$ converges strongly in, say, $L^2_\#(Q_T)$ to $\boldsymbol\varphi(t,x,z-\eta(t,x))$, since $\eta_m \to \eta$ uniformly in $\Gamma_T$. The same property implies also
\begin{equation*}
    \bb{u}(t,x,z+\eta_m(t,x)) - \bb{u}(t,x,z+\eta(t,x)) \to 0 \quad \text{ a.e. in } Q_T.
\end{equation*}
This proves that $\tilde{\bb{u}}_m \weak \tilde{\bb{u}}$ weakly in $L^2_\#(0,T;H^1_\#(\Omega))$ and the claim of the Lemma follows.
\end{proof}

\noindent We pass to the limit $m \to \infty$ in \eqref{APP1}-\eqref{APP3}. We denote by $(\rho,\bb{u},\eta)$ the limit of $(\rho_m,\bb{u}_m,\eta_m)$. The tripple $(\rho,\bb{u}, \eta)$ fulfills
\begin{equation*}
\begin{split}
&\rho \in W_\#^{1,p}(0,T; L_\#^q(\Omega))\cap L_\#^{p}(0,T; W_\#^{2,q}(\Omega)),\ \mbox{for some }p,q\in (1,2),\\
&\bb{u}\in W_\#^{1,p}(0,T; \text{span}\{\bb{f}_i\}_{1\leq i\leq n}),\ \mbox{for any } p<\infty,\\
&\eta\in W_\#^{2,p}(0,T; \text{span}\{s_i\}_{1\leq i\leq n}),\ \mbox{for any }p<\infty.
\end{split}
\end{equation*}
 The \textbf{structure momentum equation}
    \begin{equation}\label{APP1:m}
         \int_{\Gamma_T} \eta_t  \psi_t\di x\dd t -    \int_{\Gamma_T} \eta_{xx}   \psi_{xx}\di x \dd t -\int_{\Gamma_T}  \eta_{tx}  \psi_x\di x \dd t -\int_{\Gamma_T}\frac{\eta_t  - \bb{v}\cdot\bb{e}_2}\varepsilon     \psi\di x \dd t=-\int_{\Gamma_T}f\psi\di x \dd t
    \end{equation}
holds for all $\psi \in C_\#^\infty(0,T; \text{span}\{s_i\}_{1\leq i\leq n})$.\\

\noindent
The \textbf{damped continuity equation}
\begin{equation}\label{APP2:m}
    \partial_t \rho + \nabla \cdot (\rho \bb{u}) -\varepsilon\Delta \rho +\varepsilon \rho=\varepsilon M,
\end{equation}
holds almost everywhere in $Q_T$.\\

\noindent
The \textbf{fluid momentum equation}
\begin{multline}
    \delta\int_{Q_T}\bb{u}\cdot\partial_t\boldsymbol\varphi\di \bb{y} \dd t + \int_{Q_T} \rho \bb{u} \cdot \partial_t \boldsymbol\varphi\di \bb{y} \dd t + \int_{Q_T} \rho \bb{u}\otimes \bb{u}:\nabla\boldsymbol\varphi\di \bb{y} \dd t +  \int_{Q_T} (\rho^\gamma+\delta\rho^a) \nabla\cdot\boldsymbol\varphi\di \bb{y} \dd t\\
    - \int_{Q_T} \mathbb{S}(\nabla \bb{u}):\nabla\boldsymbol\varphi\di \bb{y} \dd t
    -\delta \int_{Q_T} |\bb{u}|^2\bb{u}\cdot\boldsymbol\varphi\di \bb{y} \dd t-\varepsilon\int_{Q_T}\nabla \rho\otimes\boldsymbol\varphi:\nabla \bb{u}\di \bb{y} \dd t \\
    + \frac\varepsilon2\int_{Q_T}(M-\rho)\bb{u}\cdot\boldsymbol\varphi\di \bb{y} \dd t
    -\int_{\Gamma_T}\frac{\bb{v}-\eta_t\bb{e}_2 }\varepsilon\cdot \boldsymbol\psi\di x \dd t=-\int_{Q_T}\rho\bb{F}_\delta\cdot\boldsymbol\varphi\di \bb{y} \dd t
    \label{APP3:m}
\end{multline}
holds for all $\boldsymbol\varphi\in C_\#^\infty(0,T; \text{span}\{\bb{f}_i\}_{1\leq i\leq n})$, where $\boldsymbol\psi=\gamma_{|\hat{\Gamma}^\eta}\boldsymbol\varphi$ and $\bb{v}=\gamma_{|\hat{\Gamma}^\eta}\bb{u}$ in both \eqref{APP1:m} and \eqref{APP3:m}.

\subsection{Uniform estimates independent of $n$}

First, we take $\phi \in C^\infty_{\#}(0,T)$ and choose $\psi=\phi\eta_t$ in \eqref{APP1:m}, then multiply \eqref{APP2:m} with $\frac{\gamma}{\gamma-1}\phi\rho^{\gamma-1}$, then $\frac{\delta a}{a-1}\phi\rho^{a-1}$ and $\frac{1}{2}\phi|\bb{u}|^2$ and finally choose $\boldsymbol\varphi=\phi\bb{u}$ in \eqref{APP3:m}, and then sum up these identities to obtain
\begin{multline}
     - \int_0^T\phi_t(t) E_\delta(t)\di t  +\int_{Q_T} \phi\mathbb{S}(\nabla \bb{u}):\nabla \bb{u}\di \bb{y} \dd t +\delta \int_{Q_T}\phi|\bb{u}|^4\di \bb{y} \dd t+\int_{\Gamma_T}\phi|\eta_{tx} |^2\di x \dd t\\
     + \varepsilon\gamma\int_{Q_T} \phi\rho^{\gamma-2}|\nabla \rho|^2\di \bb{y} \dd t+ \frac{\varepsilon\gamma}{\gamma-1} \int_{Q_T}\phi\rho^{\gamma}\di \bb{y} \dd t
     + \varepsilon\delta a\int_{Q_T} \phi\rho^{a-2}|\nabla \rho|^2\di \bb{y} \dd t \\
     +\frac{\varepsilon\delta a}{a-1} \int_{Q_T}\phi\rho^{a}\di \bb{y} \dd t
     + \frac1{\varepsilon}\int_{\Gamma_T} \phi|\bb{v} - \eta_t \bb{e}_2|^2\di x \dd t =\\
     = \int_{\Gamma_T} \phi f\eta_t\di x \dd t  + \int_{Q_T}\phi \rho\bb{u}\cdot\bb{F}_\delta\di \bb{y} \dd t+ \varepsilon\int_{Q_T} M\frac{\gamma}{\gamma-1}\phi\rho^{\gamma-1}\di \bb{y} \dd t+\varepsilon\delta\int_{Q_T} M\frac{a}{a-1}\phi\rho^{a-1}\di \bb{y} \dd t\label{EEtestfction}
\end{multline}
where
\begin{equation}
    E_{\delta}(t):=\int_{\Omega}\left( \frac12\rho|\bb{u}|^2+\frac{\delta}{2}|\bb{u}|^2+ \frac{1}{\gamma-1}\rho^\gamma+\frac{\delta}{a-1}\rho^a\right)(t)\di \bb{y} +\int_\Gamma\left( \frac12|\eta_{t}|^2+\frac12|\eta_{xx}|^2\right)(t)\di x.\label{Edelta.def}
\end{equation}
Choose $\phi=1$ to get
\begin{multline*}
    \int_{Q_T} \mathbb{S}(\nabla \bb{u}):\nabla \bb{u}\di \bb{y} \dd t+\delta \int_{Q_T}|\bb{u}|^4\di \bb{y} \dd t+\int_{\Gamma_T}|\eta_{tx} |^2\di x \dd t + \varepsilon\gamma\int_{Q_T} \rho^{\gamma-2}|\nabla \rho|^2\di \bb{y} \dd t \\
    + \frac{\varepsilon\gamma}{\gamma-1} \int_{Q_T}\rho^{\gamma} \di \bb{y} \dd t
     + \varepsilon\delta a\int_{Q_T} \rho^{a-2}|\nabla \rho|^2\di \bb{y} \dd t+\frac{\varepsilon\delta a}{a-1} \int_{Q_T}\rho^{a}\di \bb{y} \dd t
     + \frac1{\varepsilon}\int_{\Gamma_T} |\bb{v} - \eta_t \bb{e}_2|^2 \di x \dd t = \\
     =\int_{\Gamma_T} f\eta_t \di x \dd t + \int_{Q_T} \rho\bb{u}\cdot\bb{F}_\delta\di \bb{y} \dd t+ \varepsilon\int_{Q_T} M\frac{\gamma}{\gamma-1}\rho^{\gamma-1}\di \bb{y} \dd t+\varepsilon\delta\int_{Q_T} M\frac{a}{a-1}\rho^{a-1}\di \bb{y} \dd t. 
\end{multline*}
We deduce similarly to \eqref{appen1}
\begin{multline}
    \int_{Q_T} \mathbb{S}(\nabla \bb{u}):\nabla \bb{u}\di \bb{y} \dd t+\delta \int_{Q_T}|\bb{u}|^4\di \bb{y} \dd t+\int_{\Gamma_T}|\eta_{tx} |^2\di x \dd t + \varepsilon\gamma\int_{Q_T} \rho^{\gamma-2}|\nabla \rho|^2\di \bb{y} \dd t \\
    + \frac{\varepsilon\gamma}{\gamma-1} \int_{Q_T}\rho^{\gamma} \di \bb{y} \dd t
     + \varepsilon\delta a\int_{Q_T} \rho^{a-2}|\nabla \rho|^2\di \bb{y} \dd t+\frac{\varepsilon\delta a}{a-1} \int_{Q_T}\rho^{a}\di \bb{y} \dd t
     \\
     + \frac1{\varepsilon}\int_{\Gamma_T} |\bb{v} - \eta_t \bb{e}_2|^2\di x \dd t \leq C(\varepsilon,\delta). \label{another:important1}
\end{multline}
Next, we take a sequence of $\phi_k \to \chi_{[s,t]}$, we integrate over $(0,T)$ w.r.t. $s$ and take a supremum over $t$ to deduce
\begin{multline}\label{very:important2}
    \sup\limits_{t\in (0,T)} E_\delta(t) \leq \frac1T\int_0^T E_\delta(s)\di s + \int_{\Gamma_T} | f\eta_t | \di x \dd \tau \\
    + \int_{Q_T} | \rho\bb{u}\cdot\bb{F}_\delta | \di \bb{y} \dd \tau+ \varepsilon M\int_{Q_T} \left( \frac{\gamma}{\gamma-1}\rho^{\gamma-1} + \frac{\delta a}{a-1}\rho^{a-1}\right)\di \bb{y} \dd \tau.
\end{multline}
The last four terms can be bounded as in \eqref{appen1}. Moreover,  \eqref{another:important1} implies
\begin{equation*}
    \int_{Q_T} \left( \frac12\rho|\bb{u}|^2+\frac{\delta}{2}|\bb{u}|^2+ \frac{1}{\gamma-1}\rho^\gamma+\frac{\delta}{a-1}\rho^a \right)\di \bb{y} \dd t+ \int_{\Gamma_T}\frac12|\eta_t|^2\di x \dd t \leq C(\varepsilon,\delta).
\end{equation*}
We choose  $\psi=\eta$ in \eqref{APP1:m} to obtain $\int_{\Gamma_T} |\eta_{xx}|^2 \leq C(\varepsilon,\delta)$. Thus,  \eqref{very:important2} and previous estimates yield
\begin{equation}\label{another:important2}
    \sup\limits_{t\in (0,T)} E_\delta(t) \leq C(\varepsilon,\delta).
\end{equation}
We showed that \eqref{cor1} still holds and moreover we have additional bounds independent of $n \in \mathbb{N}$ from \eqref{another:important2}, namely
\begin{equation}
\begin{split}
    \| \eta_{xx} \|_{L^\infty(0,T;L^2(\Gamma))}&\leq C(\varepsilon,\delta), \\
    \| \eta_{t} \|_{L^\infty(0,T;L^2(\Gamma))}&\leq C(\varepsilon,\delta), \\
    \|\bb{u}\|_{L^\infty(0,T;L^2(\Omega))}&\leq C(\varepsilon,\delta), \\
    \|\sqrt{\rho}\bb{u} \|_{L^\infty(0,T;L^2(\Omega))}&\leq C(\varepsilon,\delta),\\
    \|\rho\|_{L^\infty(0,T;L^a(\Omega))}&\leq C(\varepsilon,\delta).
    \label{cor12}
\end{split}
\end{equation}
\section{Spatial basis limit $n\to \infty$}
Denote the solution obtained in previous section as $(\rho_n,\bb{u}_n,\eta_n)$. The uniform bounds \eqref{cor1} and \eqref{cor12} give rise to convergences 
\begin{equation*}
\begin{split}
    &\rho_n \weak \rho \quad \text{ weakly}^* \text{ in } L^\infty_\#(0,T;L^a_\#(\Omega)) \quad \text{ and weakly in } W^{1,p}_\#(0,T;L^q_\#(\Omega)) \cap L^p_\#(0,T;W^{2,q}_\#(\Omega)), \\
    &\bb{u}_n \weak \bb{u} \quad \text { weakly}^* \text{ in } L^\infty_\#(0,T;L^2_\#(\Omega)) \quad \text{ and weakly in } L^2_\#(0,T;H^1_\#(\Omega)), \\
    &\eta_n \weak \eta \quad \text{ weakly}^* \text{ in } L^\infty_\#(0,T;H^2_{\#}(\Gamma)) \quad \text{ and weakly in } H^1_\#(0,T;H^1_{\#,0}(\Gamma)),
\end{split}
\end{equation*}
for some $p,q \in (1,2)$. Our goal now is to pass to the limit $n\to\infty$ in \eqref{APP1:m}, \eqref{APP2:m}, \eqref{APP3:m} and \eqref{EEtestfction}.

\subsection{Limit in the structure momentum equation}
First, $\eqref{APP1:m}$ is a linear equation and thus the weak convergence is sufficient to claim
\begin{equation}\label{APP1:m2}
         \int_{\Gamma_T} \eta_t  \psi_t\di x \dd t -    \int_{\Gamma_T} \eta_{xx}   \psi_{xx}\di x \dd t -\int_{\Gamma_T}  \eta_{tx}  \psi_x\di x \dd t -\int_{\Gamma_T}\frac{\eta_t  - \bb{v}\cdot\bb{e}_2}\varepsilon     \psi\di x \dd t=-\int_{\Gamma_T}f\psi\di x \dd t,
\end{equation}
for all $\psi \in C^\infty_{\#,0}(\Gamma_T)$. We have $\|\partial_{tt}\eta_n\|_{L^2(0,T; (H^2_{\#,0}(\Gamma))^*)}\leq C(\varepsilon,\delta)$ due to $\eqref{APP1:m}$. This together with  $\|\partial_{t}\eta_n\|_{L^2_\#(0,T; H^{1}(\Gamma))}\leq C(\varepsilon,\delta)$ imply that 
\begin{equation}\label{etat:strong}
    \partial_t\eta_n \to \partial_t \eta \quad \text{ strongly in } L^2_\#(\Gamma_T).
\end{equation}
We choose $\psi = \eta_n$ in $\eqref{APP1:m}$ and $\psi=\eta$ in $\eqref{APP1:m2}$ and we compare these two identities to conclude
\begin{equation}\label{etaxx:strong}
   \int_{\Gamma_T} |\partial_{xx}\eta_n|^2\di x \dd t \to \int_{\Gamma_T} |\partial_{xx} \eta|^2\di x \dd t.
\end{equation}

\subsection{Limit in the continuity equation}
We proceed to a limit in the continuity equation. Estimates \eqref{cor1} and \eqref{cor12} yield that (upon passing to a suitable subsequence)
\begin{equation}\label{APP2:m2}
    \partial_t \rho + \nabla \cdot (\rho \bb{u}) -\varepsilon\Delta \rho +\varepsilon \rho=\varepsilon M
\end{equation}
almost everywhere in $Q_T$. We multiply \eqref{APP2:m} by $\rho_n$, integrate the resulting equation over $Q_T$ and we pass to the limit $n\to \infty$. We compare the result with \eqref{APP2:m2} multiplied by $\rho$ and integrated over $Q_T$. We deduce
\begin{equation*}
    \int_{Q_T}|\nabla\rho_n|^2\di \bb{y} \dd t \to \int_{Q_T}|\nabla\rho|^2\di \bb{y} \dd t
\end{equation*}
so
\begin{equation}\label{rho:strong}
    \nabla\rho_n \to \nabla\rho \quad \text{ strongly in } L^2(Q_T).
\end{equation}

\subsection{Limit in the fluid momentum equation}\label{s:Momn}
We start with the observation that bounds \eqref{cor1} allow to bound $\rho\bb{u}$ in $L^\frac{20}{9}(Q_T)$, which in turn implies $\|\nabla\rho_n\|_{L^\frac{20}{9}(Q_T)} \leq C(\ep,\delta)$. Consequently, we use \eqref{APP3}, to obtain
\begin{equation*}
    \|\partial_t((\delta+\rho_n)\bb{u}_n)\|_{(L^{20}_\#(0,T; W^{2,p}_\#(\Omega)))^*} \leq C(\varepsilon,\delta)
\end{equation*}
for some $p > 2$. Moreover, uniform bounds yield $\|(\delta+\rho_n)\bb{u}_n\|_{L^\infty(0,T; L^{\frac{2a}{a+1}}(\Omega))} \leq C(\varepsilon,\delta)$ and we infer $\|(\delta+\rho_n)\bb{u}_n\|_{L^\infty_\#(0,T; (W^{s,2}_\#(\Omega))^*)} \leq C(\varepsilon,\delta)$ for some $s<1$. This however means that
\begin{equation}\label{eq:strongnegative}
(\delta+\rho_n)\bb{u}_n \to (\delta+\rho)\bb{u} \quad \text{ strongly in }    L^\infty_\#(0,T; (W^{s',2}_\#(\Omega))^*)
\end{equation}
for some $s<s'<1$, and consequently by the weak convergence $\bb{u}_n\rightharpoonup \bb{u}$ in $L^2_\#(0,T;H^1_\#(\Omega))$
\begin{equation}\label{eq:convtermlimit}
    (\rho_n+\delta)\bb{u}_n \otimes \bb{u}_n \rightharpoonup (\rho+\delta)\bb{u} \otimes \bb{u} \quad \text{ weakly in } L^p_\#(Q_T) \text{ for some } p>1.
\end{equation}
Since $0\leq \frac{\rho_n}{\rho_n+\delta}<1$ and $\rho_n \to \rho$ a.e. in $Q_T$, one concludes that $\frac{\rho_n}{\rho_n+\delta} \to \frac{\rho}{\rho+\delta}$ in $L^q_\#(Q_T)$ for any $q\in [1,\infty)$ so
\begin{equation*}
    \frac{\rho_n}{\rho_n+\delta}(\rho_n+\delta)\bb{u}_n \otimes \bb{u}_n=\rho_n\bb{u}_n \otimes \bb{u}_n \rightharpoonup \rho\bb{u} \otimes \bb{u} \quad \text{in } L^1_\#(Q_T).
\end{equation*}
The weak convergence $\bb{u}_n\rightharpoonup \bb{u}$ in $L^2_\#(0,T;H^1_\#(\Omega))$ and the strong convergence of $\nabla \rho_n$ in $L^2_\#(Q_T)$ obtained in \eqref{rho:strong} yield 
\begin{equation*}
    \int_{Q_T}\nabla \rho_n\otimes\boldsymbol\varphi:\nabla \bb{u}_n\di \bb{y} \dd t \to \int_{Q_T}\nabla \rho\otimes\boldsymbol\varphi:\nabla \bb{u}\di \bb{y} \dd t,
\end{equation*}
for any $\boldsymbol\varphi \in C_\#^\infty(Q_T)$. The remaining terms are dealt with in a straightforward fashion by means of uniform bounds and Lemma \ref{tracelemma} is used to pass to the limit in the trace term. Therefore, when we let $n \to \infty$ in \eqref{APP3:m} we end up with
\begin{multline}
    \delta\int_{Q_T}\bb{u}\cdot\partial_t\boldsymbol\varphi\di \bb{y} \dd t + \int_{Q_T} \rho \bb{u} \cdot \partial_t \boldsymbol\varphi\di \bb{y} \dd t + \int_{Q_T} \rho \bb{u}\otimes \bb{u}:\nabla\boldsymbol\varphi\di \bb{y} \dd t +  \int_{Q_T} (\rho^\gamma+\delta\rho^a) \nabla\cdot\boldsymbol\varphi\di \bb{y} \dd t\\
    - \int_{Q_T} \mathbb{S}(\nabla \bb{u}):\nabla\boldsymbol\varphi\di \bb{y} \dd t
    -\delta \int_{Q_T} |\bb{u}|^2\bb{u}\cdot\boldsymbol\varphi\di \bb{y} \dd t-\varepsilon\int_{Q_T}\nabla \rho\otimes\boldsymbol\varphi:\nabla \bb{u}\di \bb{y} \dd t\\
    + \frac\varepsilon2\int_{Q_T}(M-\rho)\bb{u}\cdot\boldsymbol\varphi\di \bb{y} \dd t
    -\int_{\Gamma_T}\frac{\bb{v}-\eta_t\bb{e}_2 }\varepsilon\cdot \boldsymbol\psi\di x \dd t=\int_{Q_T}\rho\bb{F}_\delta\cdot\boldsymbol\varphi\di \bb{y} \dd t,
    \label{APP3:m2}
\end{multline}
for all $\boldsymbol\varphi \in C_{\#}^\infty(Q_T)$ and $\boldsymbol\psi\in C_{\#}^\infty(\Gamma_T)$ such that $\boldsymbol\varphi (t,x,\hat{\eta}(t,x))=\boldsymbol\psi(t,x)$ on $\Gamma_T$, where $\bb{v} =  \gamma_{|\hat{\Gamma}^\eta} \bb{u}$.


\subsection{Limit in the energy inequality}
The information gathered above is clearly sufficient to pass to the limit in all terms on the right hand side of \eqref{EEtestfction}. In order to pass to the limit on the left hand side we first note that \eqref{etat:strong}, \eqref{etaxx:strong} together with \eqref{eq:convtermlimit} and the information about the sequence of densities allows us to pass to the limit in the first term on the left hand side of \eqref{EEtestfction}. Finally, we assume that $\phi \in C^\infty_\#(0,T)$ satisfies moreover $\phi \geq 0$ and we use weak lower semicontinuity of convex functions to deduce that in the limit, \eqref{EEtestfction} holds as an inequality
\begin{multline}
     - \int_0^T\phi_t(t) E_\delta(t) \di t +\int_{Q_T} \phi\mathbb{S}(\nabla \bb{u}):\nabla \bb{u}\di \bb{y} \dd t+\delta \int_{Q_T}\phi|\bb{u}|^4 \di \bb{y} \dd t + \int_{\Gamma_T}\phi|\eta_{tx} |^2\di x \dd t \\
     + \varepsilon\gamma\int_{Q_T} \phi\rho^{\gamma-2}|\nabla \rho|^2\di \bb{y} \dd t
     + \frac{\varepsilon\gamma}{\gamma-1} \int_{Q_T}\phi\rho^{\gamma}\di \bb{y} \dd t+ \varepsilon\delta a\int_{Q_T} \phi\rho^{a-2}|\nabla \rho|^2\di \bb{y} \dd t
     \\
     +\frac{\varepsilon\delta a}{a-1} \int_{Q_T}\phi\rho^{a}\di \bb{y} \dd t
     + \frac1{\varepsilon}\int_{\Gamma_T} \phi|\bb{v} - \eta_t \bb{e}_2|^2\di x \dd t 
     \leq \int_{\Gamma_T} \phi f\eta_t\di x \dd t 
     \\
     + \int_{Q_T}\phi \rho\bb{u}\cdot\bb{F}_\delta\di \bb{y} \dd t+ \varepsilon\int_{Q_T} M\frac{\gamma}{\gamma-1}\phi\rho^{\gamma-1}\di \bb{y} \dd t+\varepsilon\delta\int_{Q_T} M\frac{a}{a-1}\phi\rho^{a-1}\di \bb{y} \dd t\label{EEtestfction_lim1}
\end{multline}
where $E_\delta$ is defined by \eqref{Edelta.def}.



\subsection{Uniform bounds independent of $\varepsilon$}

We use the energy inequality \eqref{EEtestfction_lim1} to deduce estimates of $(\rho,\bb{u},\eta)$ independent of $\ep$. We start by taking $\phi = 1$ in \eqref{EEtestfction_lim1} to get 
 \begin{multline}
     \int_{Q_T} \mathbb{S}(\nabla \bb{u}):\nabla \bb{u}\di \bb{y} \dd t+\delta \int_{Q_T}|\bb{u}|^4\di \bb{y} \dd t+\int_{\Gamma_T}|\eta_{tx} |^2\di x \dd t + \varepsilon\gamma\int_{Q_T} \rho^{\gamma-2}|\nabla \rho|^2\di \bb{y} \dd t\\
     + \frac{\varepsilon\gamma}{\gamma-1} \int_{Q_T}\rho^{\gamma}\di \bb{y} \dd t 
       + \varepsilon\delta a\int_{Q_T} \rho^{a-2}|\nabla \rho|^2\di \bb{y} \dd t+\frac{\varepsilon\delta a}{a-1} \int_{Q_T}\rho^{a}\di \bb{y} \dd t
      + \frac1{\varepsilon}\int_{\Gamma_T} |\bb{v} - \eta_t\bb{e}_2|^2\di x \dd t \\
      \leq \int_{\Gamma_T} f\eta_t\di x \dd t  + \int_{Q_T} \rho\bb{u}\cdot\bb{F}_\delta\di \bb{y} \dd t+ \varepsilon\int_{Q_T} M\frac{\gamma}{\gamma-1}\rho^{\gamma-1}\di \bb{y} \dd t+\varepsilon\delta\int_{Q_T} M\frac{a}{a-1}\rho^{a-1}\di \bb{y} \dd t. \label{very:important1a}
 \end{multline}
The estimates here need to be more delicate than in the previous section as we no longer have directly information about the density independent of $\ep$ on the left hand side of \eqref{very:important1a}. Therefore we introduce (recall \eqref{Edelta.def})
 \begin{equation}\label{Edeltadef}
     \mathcal{E}_\delta :=\sup\limits_{t\in (0,T)} E_\delta(t). 
 \end{equation}
We take $\phi \to \chi_{[s,t]}$ in \eqref{EEtestfction_lim1}, we integrate over $(0,T)$ with respect to $s$ and finally we take the supremum over $t$ to get
 \begin{multline}\label{very:important2a}
     \mathcal{E}_\delta \leq \frac1T\int_0^T E_\delta(s) \di s+ \int_{\Gamma_T}  f\eta_t\di x \dd t  + \int_{Q_T} \rho\bb{u}\cdot\bb{F}_\delta\di \bb{y} \dd t\\
     + \varepsilon\int_{Q_T} M\frac{\gamma}{\gamma-1}\rho^{\gamma-1}\di \bb{y} \dd t+\varepsilon\delta\int_{Q_T} M\frac{a}{a-1}\rho^{a-1}\di \bb{y} \dd t.
 \end{multline}
Our goal is therefore to bound the terms on the right-hand sides of \eqref{very:important1a} and \eqref{very:important2a}. The first, third and fourth terms on the right-hand side of \eqref{very:important1a} can be absorbed as in \eqref{appen1}. The second term has to be estimated in a different way. Let $p > 1$ be small and let $q = \frac{p}{p-1}$. We have
\begin{multline*}
    \int_{Q_T}\rho\bb{u}\cdot\bb{F}_\delta\di \bb{y} \dd t\leq C \|\rho\|_{L^\infty(0,T; L^p(\Omega))} \|\bb{u}\|_{L^2(0,T;L^q(\Omega))} \leq C \|\rho\|_{L^\infty(0,T; L^p(\Omega))} \|\bb{u}\|_{L^2(0,T;H^1(\Omega))} \\
    \leq C(s,\delta)(1+ \mathcal{E}_\delta^s) + \frac{\delta}2 \left( \int_{Q_T} \mathbb{S}(\nabla \bb{u}):\nabla \bb{u}\di \bb{y} \dd t+ \int_{Q_T}|\bb{u}|^4\di \bb{y} \dd t \right)
\end{multline*}
for $s>0$ as small as we want, where we interpolated $L^p$ between $L^1$ and $L^a$. Provided $\delta<1$, these terms can be absorbed so it leads to 
\begin{multline*}
     \int_{Q_T} \mathbb{S}(\nabla \bb{u}):\nabla \bb{u}\di \bb{y} \dd t+\delta \int_{Q_T}|\bb{u}|^4\di \bb{y} \dd t+\int_{\Gamma_T}|\eta_{tx} |^2\di x \dd t + \varepsilon\gamma\int_{Q_T} \rho^{\gamma-2}|\nabla \rho|^2\di \bb{y} \dd t\\
     + \frac{\varepsilon\gamma}{\gamma-1} \int_{Q_T}\rho^{\gamma} \di \bb{y} \dd t
      + \varepsilon\delta a\int_{Q_T} \rho^{a-2}|\nabla \rho|^2\di \bb{y} \dd t+\frac{\varepsilon\delta a}{a-1} \int_{Q_T}\rho^{a}\di \bb{y} \dd t
     + \frac1{\varepsilon}\int_{\Gamma_T} |\bb{v} - \eta_t\bb{e}_2|^2\di x \dd t
     \\
     \leq C(s,\delta)(1+ \mathcal{E}_\delta^s).
\end{multline*}
The last four terms on the right hand side of \eqref{very:important2a} are treated the same way, hence it remains to show
\begin{equation}\label{eq:circular_epsilon}
    \int_0^T E_\delta(s)\di s \leq C(1+\mathcal{E}_\delta^{\beta})
\end{equation}
for some $\beta < 1$.

We observe that
\begin{equation*}
    \int_{Q_T} \frac{1}{2}(\rho + \delta)|\bb{u}|^2\di \bb{y} \dd t + \int_{\Gamma_T} \frac12 |\eta_t|^2\di x \dd t \leq C(s,\delta)(1 + \mathcal{E}_\delta^{\frac{s}{2} + \frac{s}{a}} + \mathcal{E}_\delta^s).
\end{equation*}
We multiply \eqref{APP2:m} by $\rho$ and integrate over $Q_T$ to get
\begin{multline}
    \varepsilon\int(\rho^2+|\nabla\rho|^2)\di \bb{y} \dd t = \int_{Q_T} -\frac12 \rho^2 \nabla\cdot\bb{u}\di \bb{y} \dd t + \int_{Q_T} \ep M \rho\di \bb{y} \dd t \\
    \leq \left(\int_{Q_T} \rho^{4}\di \bb{y} \dd t\right)^{\frac12} \| \bb{u}\|_{L^2(0,T; H^1(\Omega))}+C  
    \leq \left(\int_{Q_T} \rho^{a}\di \bb{y} \dd t\right)^{\frac2a} C(s,\delta)(1+ \mathcal{E}_\delta^s) \leq C(s,\delta)(1+ \mathcal{E}_\delta^{s+\frac2a}). \label{eq:epgradrho}
\end{multline}
Next, we choose $\psi = \eta$ in \eqref{APP1:m} and sum up the resulting equation with \eqref{APP3:m} with the choice $\boldsymbol\varphi=\eta\bb{e}_2$. Most of the calculations can be done in the same way as in Section \ref{eta:est:sec}, however we need to estimate several additional terms multiplied by approximation parameters, namely
\begin{equation*}
    \left|\int_{Q_T}\delta\bb{u}\cdot\eta_t\bb{e}_2\di \bb{y} \dd t\right| \leq C(\delta)\|\bb{u}\|_{L^4(Q_T)}\|\eta_t\|_{L^2(0,T;L^\infty(\Gamma))} \leq C(s,\delta)(1+\mathcal{E}_\delta^{\frac34 s}),
\end{equation*}
\begin{multline*}
    \left|\int_{Q_T}\delta|\bb{u}|^2\bb{u}\cdot\eta\bb{e}_2\di \bb{y} \dd t\right| \leq C(\delta)\|\bb{u}\|^3_{L^4(Q_T)}\|\eta\|_{L^4(\Gamma_T)} \leq C(s,\delta)(1+\mathcal{E}_\delta^{\frac34 s})(\|\eta_x\|_{L^2(\Gamma_T)} + \|\eta_t\|_{L^2(\Gamma_T)}) \\ 
     \leq C(s,\delta)(1+\mathcal{E}_\delta^{\frac32 s}) + \frac{1}{16}\|\eta_{xx}\|^2_{L^2(\Gamma_T)},
\end{multline*}
\begin{multline*}
    \left|\frac{\varepsilon}{2}\int_{Q_T}(M-\rho)\bb{u}\cdot\eta\bb{e}_2\di \bb{y} \dd t\right| \leq C(\delta)(1+\|\rho\|_{L^\infty(0,T;L^p(\Omega))})\|\bb{u}\|_{L^2(0,T;L^q(\Omega))}\|\eta\|_{L^2(0,T;L^\infty(\Gamma))} \\
     \leq  C(s,\delta)(1+\mathcal{E}_\delta^{2s}) + \frac{1}{16}\|\eta_{xx}\|^2_{L^2(\Gamma_T)},
\end{multline*}
and
\begin{multline*}
    \left|\varepsilon\int_{Q_T}\nabla\rho\otimes(\eta\bb{e}_2):\nabla\bb{u}\di \bb{y} \dd t\right| \leq C(\delta)\|\sqrt{\varepsilon}\nabla\rho\|_{L^2(Q_T)}\|\nabla\bb{u}\|_{L^2(Q_T)}\|\eta\|_{L^\infty(\Gamma_T)} \\
     \leq  C(s,\delta)(1+\mathcal{E}_\delta^{2s+\frac{2}{a}}) + \frac{1}{16}\|\eta_{xx}\|^2_{L^2(\Gamma_T)}.
\end{multline*}
Eventually we end up with the estimate
\begin{equation*}
    \int_{\Gamma_T} |\eta_{xx}|^2\di x \dd t\leq C(s,\delta)(1+ \mathcal{E}_\delta^{s'}),
\end{equation*}
for some $0<s'<1$. 

It remains to show
\begin{equation}\label{eq:Bog_epsilon}
    \int_{\Omega}\left(\frac{1}{\gamma-1}\rho^\gamma+\frac{\delta}{a-1}\rho^a\right)\di \bb{y} \dd t \leq C(s,\delta)(1+ \mathcal{E}_\delta^{s''}),
\end{equation}
for some $0<s''<1$ similarly to Section \ref{Bog:section}. To this end we use  $\boldsymbol\varphi_h$ defined in \eqref{eq:varphi_h_def} as a test function in \eqref{APP3:m2}. As above in the estimate of second spatial derivatives of $\eta$, we obtain four more terms to estimate. The term $\delta\bb{u}\cdot\partial_t\boldsymbol\varphi_h$ is handled similarly as $\rho\bb{u}\cdot\partial_t\boldsymbol\varphi_h$. The remaining three additional terms are easy to handle due to the estimate
\begin{equation*}
    \|\boldsymbol\varphi_h\|_{L^\infty(Q_T)} \leq \left\|\mathcal{B}_\Omega\left[\rho^\alpha-\int_\Omega\rho^\alpha\di x \right]\right\|_{L^\infty(Q_T)} \leq C
\end{equation*}
which follows from \eqref{eq:Bogovskii_infty}. Therefore
\begin{equation*}
    \left|\int_{Q_T}\delta|\bb{u}|^2\bb{u}\cdot\boldsymbol\varphi_h\di \bb{y} \dd t\right| \leq C(\delta)\|\bb{u}\|^3_{L^4(Q_T)} \leq C(s,\delta)(1+\mathcal{E}_\delta^{\frac34 s}),
\end{equation*}
\begin{equation*}
    \left|\frac{\varepsilon}{2}\int_{Q_T}(M-\rho)\bb{u}\cdot\boldsymbol\varphi_h\di \bb{y} \dd t\right| \leq C(\delta)(1+\|\rho\|_{L^\infty(0,T;L^p(\Omega))})\|\bb{u}\|_{L^2(0,T;L^q(\Omega))} \leq  C(s,\delta)(1+\mathcal{E}_\delta^{s}),
\end{equation*}
and
\begin{equation*}
    \left|\varepsilon\int_{Q_T}\nabla\rho\otimes\boldsymbol\varphi_h:\nabla\bb{u}\di \bb{y} \dd t\right| \leq C(\delta)\|\sqrt{\varepsilon}\nabla\rho\|_{L^2(Q_T)}\|\nabla\bb{u}\|_{L^2(Q_T)} \leq C(s,\delta)(1+\mathcal{E}_\delta^{s+\frac{1}{a}}).
\end{equation*}
In the second part of this procedure we use the test function $\boldsymbol\varphi = \varphi_h\bb{e}_2$ in \eqref{APP3:m2} with $\varphi_h$ defined in \eqref{eq:varphi_h_def2}. The estimates are again either similar to those in Section \ref{Bog:section} or to those presented above and we recover \eqref{eq:Bog_epsilon}. This however means that \eqref{eq:circular_epsilon} is proved which yields 
\begin{equation}
     \mathcal{E}_\delta \leq C(\delta), \label{uniform:est2}
\end{equation}
and
\begin{multline}
    \int_{Q_T} \mathbb{S}(\nabla \bb{u}):\nabla \bb{u}\di \bb{y} \dd t+\delta \int_{Q_T}|\bb{u}|^4\di \bb{y} \dd t+\int_{\Gamma_T}|\eta_{tx} |^2\di x \dd t
    + \varepsilon\gamma\int_{Q_T} \rho^{\gamma-2}|\nabla \rho|^2\di \bb{y} \dd t\\
    + \frac{\varepsilon\gamma}{\gamma-1} \int_{Q_T}\rho^{\gamma}\di \bb{y} \dd t+ \varepsilon a\delta\int_{Q_T} \rho^{a-2}|\nabla \rho|^2\di \bb{y} \dd t+\frac{\varepsilon\delta a}{a-1} \int_{Q_T}\rho^{a+1}\di \bb{y} \dd t\\
     + \frac1{\varepsilon}\int_{\Gamma_T} |\bb{v} - \eta_t\bb{e}_2|^2\di x \dd t \leq  C(\delta). \label{uniform:est1}
\end{multline}

\subsection{Coupled back momentum equation} 
We sum up the momentum equation \eqref{APP3:m2} for test functions $(\boldsymbol\varphi,\psi)$ and the  structure momentum equation \eqref{APP1:m2} for test function $\psi$. This way the penalization terms get cancelled and we obtain that $(\rho,\bb{u},\eta)$ satisfy the coupled momentum equation
\begin{multline}
    \delta\int_{Q_T}\bb{u}\cdot\partial_t\boldsymbol\varphi\di \bb{y} \dd t + \int_{Q_T} \rho \bb{u} \cdot \partial_t \boldsymbol\varphi\di \bb{y} \dd t + \int_{Q_T} \rho \bb{u}\otimes \bb{u}:\nabla\boldsymbol\varphi\di \bb{y} \dd t +  \int_{Q_T} (\rho^\gamma+\delta\rho^a) \nabla\cdot\boldsymbol\varphi\di \bb{y} \dd t\\
    - \int_{Q_T} \mathbb{S}(\nabla \bb{u}):\nabla\boldsymbol\varphi\di \bb{y} \dd t-\delta \int_{Q_T} |\bb{u}|^2\bb{u}\cdot\boldsymbol\varphi\di \bb{y} \dd t-\varepsilon\int_{Q_T}\nabla \rho\otimes\boldsymbol\varphi:\nabla \bb{u}\di \bb{y} \dd t+ \frac\varepsilon2\int_{Q_T}(M-\rho)\bb{u}\cdot\boldsymbol\varphi \di \bb{y} \dd t\\
    -  \int_{\Gamma_T} \eta_t  \psi_t \di x \dd t-    \int_{\Gamma_T} \eta_{xx}   \psi_{xx}\di x \dd t -\int_{\Gamma_T}  \eta_{tx}  \psi_x\di x \dd t = -\int_{\Gamma_T}f\psi\di x \dd t-\int_{Q_T}\rho\bb{F}_\delta\cdot\boldsymbol\varphi\di \bb{y} \dd t, \label{coupledback:mom}
\end{multline}
which holds for all $\boldsymbol\varphi \in C_{\#}^\infty(Q_T)$ and $\psi\in C_{\#,0}^\infty(\Gamma_T)$ such that $\boldsymbol\varphi (t,x,\hat{\eta}(t,x))=\psi(t,x)\bb{e}_2$ on $\Gamma_T$. Note however, that at this point, the problem is still not fully coupled since we cannot ensure that $\eta_t\bb{e}_2 = \gamma_{|\Gamma^\eta}\bb{u}$.

\subsection{Improved estimate of $\eta_{xx}$}\label{sec:improvedeta}
The following approach comes from \cite{MuhaSch}, where the improved regularity of displacement was obtained for the interaction problem between an incompressible viscous fluid and a nonlinear Koiter shell (see also \cite[Theorem 2.2]{Tr} for the compressible counterpart). We start with introducing the notation $D_h^s[\eta]$ defined as
\begin{equation*}
    D_{h}^s[\eta](x):= \frac{\eta(t,x+h) - \eta(t,x)}{|h|^{s-1}h}, \quad s>0, h \in \R.
\end{equation*}
The idea is to take $s < \frac14$ and test the coupled momentum equation \eqref{coupledback:mom} with a suitable test function to obtain an estimate on $\int_{\Gamma_T}|D_{h}^s[\eta_{xx}]|^2 \di x\dd t$ independent on $h < h_0$ for some $h_0 > 0$. The integration by parts formula for $D^s_h$ holds for periodic functions, i.e.
$$
\int_{\Gamma}D^s_h[u](x)v(x)\di x  = -\int_{\Gamma} u(x) D^s_{-h}[v](x)\di x
$$
for all periodic $u,v$ such that the integrals are finite. We set
\begin{equation*}
    \psi_h(t,x) = D^s_{-h}[D^s_h[\eta(t,x)]] - \frac{1}{|h|^{2s}}(\eta(t,-h)+\eta(t,h)) =: \psi_{1,h}(t,x) - \psi_{2,h}(t)
\end{equation*}
and use $(\psi_h\bb{e}_2,\psi_h)$ as a test function couple in \eqref{coupledback:mom} (note that this is an admissible test function because $\psi_h(t,0) = 0$). This gives rise to
\begin{eqnarray*}
    -\int_{\Gamma_T} \eta_{xx} (\psi_h)_{xx}\di x\dd t = RHS,
\end{eqnarray*}
so by taking into account that $(\psi_h)_{xx}=D^s_{-h}[D^s_h[\eta(t,x)_{xx}]]$ which implies
\begin{eqnarray*}
     \int_{\Gamma_T}|D_h^s[\eta_{xx}(t,x)]|^2 \di x\dd t= -\int_{\Gamma_T} \eta_{xx} (\psi_h)_{xx}\di x\dd t,
\end{eqnarray*}
the proof will follow once we show that RHS is bounded. \\

First, note that
\begin{eqnarray}
    &&\|D^{s}_{-h}[D^s_h [\eta_t]]\|_{L^p(\Gamma)} \leq C \|\eta_{tx}\|_{L^2(\Gamma)}, \label{bound:Dh:1} \\
    &&\|D^{s}_{-h}[D^s_h [\eta_x]]\|_{L^p(\Gamma)} \leq C \|\eta_{xx}\|_{L^2(\Gamma)},\label{bound:Dh:2}
\end{eqnarray}
for any $p>1$ and $s<\frac14$ by embedding theorems (see \cite[Proposition 2]{simon} and \cite[Proposition 4.6]{triebel}). Moreover, since $||\eta_{tx}||_{L^2(\Gamma_T)} \leq C(\delta)$, we get $\eta_t \in L^2(0,T;C^{\frac 12}(\Gamma))$ and thus
\begin{eqnarray*}
    \frac{\eta_t(t,\pm h)}{|h|^{\frac 12}} = \frac{\eta_t(t,\pm h) - \eta_t(t,0)}{|h|^{\frac 12}} \in L^2(0,T)
\end{eqnarray*}
with its $L^2$-norm bounded by $C(\delta)$. This means that for $s < \frac 14$ it holds $\psi_{2,t} \in L^2(0,T)$ and $||\psi_{2,t}||_{L^2(0,T)} \leq C(\delta)$. This combined with $\eqref{bound:Dh:1}$ implies
\begin{eqnarray}
    \|D^{s}_{-h}[D^s_h [(\psi_h)_t]]\|_{L^2(0,T; L^p(\Gamma))} \leq C \|\eta_{tx}\|_{L^2(0,T;L^2(\Gamma))} \leq C(\delta) \label{bound:Dh:12},
\end{eqnarray}
while $\eqref{bound:Dh:2}$ implies
\begin{eqnarray}
     \|D^{s}_{-h}[D^s_h [(\psi_h)_{x}]]\|_{L^\infty(0,T; L^p(\Gamma))} \leq C \|\eta_{xx}\|_{L^\infty(0,T;L^2(\Gamma))} \leq C(\delta),\label{bound:Dh:22}
\end{eqnarray}
for any $p>1$ and $s<\frac14$. Finally, since $||\eta_{xx}||_{L^\infty(0,T;L^2(\Gamma))} \leq C(\delta)$ a simple first order Taylor expansion of $\eta$ yields
$$
\psi_2(t) \leq C(\delta)|h|^{1-2s} \leq C(\delta),
$$
so
\begin{eqnarray}
     \|D^{s}_{-h}[D^s_h [(\psi_h)]]\|_{L^\infty(\Gamma_T)} \leq C ( \|\eta_{xx}\|_{L^\infty(0,T;L^2(\Gamma))} + ||\psi_{2,h}||_{L^\infty}(0,T)) \leq C(\delta).\label{bound:Dh:32}
\end{eqnarray}
Now, we are ready to show that the arising terms are bounded. First, the bounds of the terms involving time derivatives of $\psi_h$ are bounded as follows
\begin{eqnarray*}
    \left|\int_{Q_T} \rho \bb{u}\cdot (\partial_t \psi_h \bb{e}_2)\di \bb{y} \dd t \right| \leq C||\rho||_{L^\infty(0,T; L^\gamma(\Omega))} ||\bb{u}||_{L^2(0,T; L^p(\Omega))}  \|D^{s}_{-h}[D^s_h [(\psi_h)_t]]\|_{L^2(0,T; L^p(\Gamma))} \leq C(\delta)
\end{eqnarray*}
for $p=\frac{2\gamma}{\gamma-1}$ by $\eqref{bound:Dh:12}$, and
\begin{eqnarray*}
    &&\delta\left|\int_{Q_T} \bb{u}\cdot (\partial_t \psi_h \bb{e}_2)\di \bb{y} \dd t \right| \leq  C\delta^{\frac34} ||\delta^{\frac14}\bb{u}||_{L^4(Q_T)}  \|D^{s}_{-h}[D^s_h [(\psi_h)_t]]\|_{L^2(\Gamma_T)} \leq C(\delta),\\
    &&\left|\int_{\Gamma_T}\eta_t (\psi_h)_t \di x \dd t \right| \leq || \eta_t||_{L^2(\Gamma_T)} \|D^{s}_{-h}[D^s_h [(\psi_h)_t]]\|_{L^2(\Gamma_T)} \leq C(\delta),
\end{eqnarray*}
by $\eqref{bound:Dh:12}$ and uniform bounds. Next,  the pressure term vanishes since $\nabla\cdot (D^s_{-h}[D^s_h[\eta]](x) \bb{e}_2) = 0$. The remaining terms all include at most one spatial derivative on $\psi_h$. Let us bound only the most "difficult" terms:
\begin{eqnarray*}
    \left|\varepsilon\int_{Q_T}\nabla \rho\otimes(\psi_h \bb{e}_2):\nabla \bb{u}\di \bb{y} \dd t \right| \leq \sqrt{\varepsilon} ||\sqrt{\varepsilon}\nabla \rho||_{L^2(Q_T)} ||\psi_h||_{L^\infty(\Gamma_T)} ||\nabla \bb{u}||_{L^2(Q_T)} \leq C(\delta)
\end{eqnarray*}
by $\eqref{bound:Dh:32}$, and
\begin{eqnarray*}
    \left| \int_{Q_T} \rho \bb{u}\otimes \bb{u}:\nabla\boldsymbol\varphi_h \di \bb{y} \dd t\right| \leq C||\rho||_{L^\infty(0,T; L^\gamma(\Omega))} ||\bb{u}||_{L^2(0,T; L^p(\Omega))}^2 ||(\psi_h)_x||_{L^\infty(0,T; L^p(\Omega))}
\end{eqnarray*}
for $p=\frac{3\gamma}{\gamma-1}$, by $\eqref{bound:Dh:22}$. The remaining terms are bounded in a similar fashion, so we conclude
\begin{eqnarray*}
    \int_{\Gamma_T}|D_{h}^s[\eta_{xx}]|^2 \leq C(\delta)
\end{eqnarray*}
and as a direct consequence of imbedding and uniform bound on $\eta$ in $L^2(0,T; H^2(\Gamma))$, one finally obtains
\begin{eqnarray}\label{improved:eta}
    || \eta||_{L^2(0,T; H^{2+s}(\Gamma))}\leq C(\delta)
\end{eqnarray}
for any $s<\frac14$.

\section{Limit $\varepsilon\to 0$}
Denote the solutions obtained in previous section as $(\rho_\varepsilon,\bb{u}_\varepsilon,\eta_\varepsilon)$. The uniform bounds \eqref{uniform:est2} and \eqref{uniform:est1} give rise to the following weak convergencies
\begin{equation*}
\begin{split}
    &\rho_\varepsilon \weak \rho \quad \text{weakly$^*$ in } L^\infty_\#(0,T;L^a_\#(\Omega)), \\
    &\bb{u}_\varepsilon \weak \bb{u} \quad \text{ weakly in }L^2_\#(0,T;H^1_\#(\Omega)), \\
    &\eta_\varepsilon \weak \eta \quad \text{weakly$^*$ in } L^\infty_\#(0,T;H^2_{\#}(\Gamma))\quad \text{and weakly in } H^1_\#(0,T;H^1_{\#,0}(\Gamma)).
\end{split}
\end{equation*}
We pass to the limit in the equations \eqref{APP2:m2}, \eqref{coupledback:mom} and the energy inequality \eqref{EEtestfction_lim1}.

\subsection{Limit in the continuity equation}

We use nowadays standard arguments for the continuity equation to get $\rho_\ep \to \rho$ in $C_w([0,T];L^a(\Omega))$ and therefore $\rho_\ep\bb{u}_\ep \weak \rho\bb{u}$ weakly in $L^\infty(0,T;L^{\frac{2a}{a+1}}(\Omega))$. Moreover, due to \eqref{eq:epgradrho} and \eqref{uniform:est2} we have $\ep\nabla\rho_\ep \to 0$ in $L^2(Q_T)$. We conclude that the limiting functions $\rho$ and $\bb{u}$ satisfy the continuity equation in the weak sense, i.e.
\begin{equation*}
     \int_{Q_T} \rho (\partial_t \varphi +\bb{u}\cdot \nabla \varphi)\di \bb{y} \dd t=0
\end{equation*}
for all $\varphi \in C^{\infty}_\#(Q_T)$. Since $\rho\in L_\#^\infty(0,T; L_\#^a(\Omega))$ and $a\geq 2$ we further get that the renormalized continuity equation is satisfied by $\rho$ and $\bb{u}$, i.e.
\begin{equation*}
     \int_{Q_T} \rho B(\rho)( \partial_t \varphi +\bb{u}\cdot \nabla \varphi)\di \bb{y} \dd t =\int_{Q_T} b(\rho)(\nabla\cdot \bb{u}) \varphi \di \bb{y} \dd t
\end{equation*}
for all functions $\varphi \in C_\#^\infty(Q_T)$ and any $b\in L^\infty (0,\infty) \cap C[0,\infty)$ such that $b(0)=0$ with $B(\rho)=B(1)+\int_1^\rho \frac{b(z)}{z^2}dz$, see i.e. \cite[Section 11.19]{FeNobook}.

\subsection{Limit in the coupled momentum equation}

The limit in the equation \eqref{coupledback:mom} is more involved. The terms integrated over $\Gamma_T$ are linear and their limits are straightforward. Regarding the terms integrated over $Q_T$, we start similarly as in Section \ref{s:Momn}, deduce from the continuity equation that
\begin{equation}\label{eq:epgrad}
    \|\ep \nabla \rho_\ep \|_{L^{\frac{20}{9}}(Q_T)} \leq C(\delta)
\end{equation}
and we use this information to estimate
\begin{equation*}
    \|\partial_t((\delta+\rho_\ep)\bb{u}_\ep) \|_{(L^{20}_\#(0,T,W^{2,p}_\#(\Omega)))^*} \leq C(\delta).
\end{equation*}
The continuity equation implies a similar estimate for the time derivative of the density, namely
\begin{equation*}
    \|\partial_t\rho_\ep \|_{(L^{\frac{10}{3}}_\#(0,T,W^{1,2}_\#(\Omega)))^*} \leq C(\delta).
\end{equation*}
Using this information and the fact that the sequence of velocities is bounded in $L^4(Q_T)$ we get in particular that
\begin{equation*}
    \left|\int_{Q_T} \partial_t\rho_\ep \bb{u}_\ep\cdot \boldsymbol\varphi\di \bb{y} \dd t\right| \leq C(\delta)
\end{equation*}
for any $\boldsymbol\varphi \in L^{20}_\#(0,T,W^{2,p}_\#(\Omega))$. Therefore we obtain
\begin{multline*}
    \delta \|\partial_t\bb{u}_\ep \|_{(L^{20}_\#(0,T,W^{2,p}_\#(\Omega)))^*} \leq \|(\delta+\rho_\ep) \partial_t\bb{u}_\ep \|_{(L^{20}_\#(0,T,W^{2,p}_\#(\Omega)))^*} \\
    \leq \|\partial_t((\delta+\rho_\ep)\bb{u}_\ep) \|_{(L^{20}_\#(0,T,W^{2,p}_\#(\Omega)))^*} + \|\bb{u}_\ep\partial_t\rho_\ep \|_{(L^{20}_\#(0,T,W^{2,p}_\#(\Omega)))^*} \leq    
    C(\delta).
\end{multline*}
This bound together with the Aubin-Lions lemma is enough to pass to the limit in the term $\delta \int_{Q_T} |\bb{u}|^2\bb{u}\cdot\boldsymbol\varphi \di \bb{y} \dd t$. We also obtain similar convergences as in \eqref{eq:strongnegative} and \eqref{eq:convtermlimit}, where we combine the latter with the fact that 
\begin{equation*}
    \bb{u}_\ep \otimes \bb{u}_\ep \to \bb{u}\otimes\bb{u} \quad \text{ in } L^p(Q_T) \text{ for some } p > 1
\end{equation*}
to pass to the limit in the convective term. 

The only remaining term without properly identified limit is the pressure term. Regarding this term, we first observe that when deriving \eqref{eq:Bog_epsilon}, we proved that $\rho_\varepsilon^a$ has a better than $L^1$ integrability in the interior of the domain $Q_T^\eta$. However, it is still possible that $\{\rho_\varepsilon\}_{\varepsilon>0}$ might generate some concentrations near the elastic boundary. We define
\begin{equation*}
    \varphi_h^\varepsilon(t,x,z):=\begin{cases}
    \frac{z-\eta_\varepsilon(t,x)}{h}, &\quad \text{ for } \eta_\varepsilon(t,x)<z<\eta_\varepsilon(t,x)+h,\\
    -\frac{1}{H-h}(z-(\eta_\varepsilon(t,x)+h))+1, &\quad \text{ for } \eta_\varepsilon(t,x)+h<z<\eta_\varepsilon(t,x)+{2H}-h,\\
    \frac{z-(\eta_\varepsilon(t,x)+2H)}{h}, &\quad \text{ for } \eta_\varepsilon(t,x)+{2H}-h<z<\eta_\varepsilon(t,x)+{2H}.
    \end{cases}
\end{equation*}
We choose $\boldsymbol\varphi=\varphi_h^\varepsilon\bb{e}_2$ in \eqref{coupledback:mom} (with $\psi=0$) and we compute similarly as in \eqref{bndest} to get
\begin{equation}\label{eq:equiintep}
    \int_0^T \int_{\{\eta<z<\eta+h\}\cup\{\eta+2H-h<z<\eta+2H\}} (\rho_\varepsilon^{\gamma}+\delta\rho_\varepsilon^a)\di \bb{y} \dd t \leq C(\delta) h^s,
\end{equation}
for some $s>0$. Indeed, to obtain this kind of estimate it is enough to observe that all arising terms have better than $L^1$ integrability in the space variable. Here we in particular use once again \eqref{eq:epgrad}.

Estimate \eqref{eq:equiintep} means that the sequence $\{\rho_\varepsilon^{\gamma}+\delta\rho_\varepsilon^a\}_{\varepsilon>0}$ is uniformly integrable so there exists its weak limit in $L^1(Q_T)$ denoted as $\overline{p_\delta(\rho)}$. In order to identify $\overline{p_\delta(\rho)}$, one can use the standard approach on compact subsets of $Q_T^\eta$ based on convergence of effective viscous flux, renormalized continuity equation and monotonicity argument (see \cite{FeNobook}) in order to conclude that 
\begin{equation*}
    \rho_\varepsilon \to \rho, \quad \text{a.e. in } Q_T.
\end{equation*}
This is enough to identify $\overline{p_\delta(\rho)}$ as $\rho^{\gamma}+\delta\rho^a$. 

Finally, let us point out that the kinematic coupling $\partial_t\eta\bb{e}_2 = \gamma_{|\Gamma^\eta}\bb{u}$ is recovered due to the bound \eqref{uniform:est1}. We have proved that the limit functions $(\rho,\bb{u},\eta)$ satisfy
 \begin{multline}
     \int_{Q_T}(\delta + \rho) \bb{u} \cdot\partial_t \boldsymbol\varphi\di \bb{y} \dd t + \int_{Q_T}(\rho \bb{u} \otimes \bb{u}):\nabla\boldsymbol\varphi\di \bb{y} \dd t +\int_{Q_T} (\rho^\gamma+\delta\rho^a) (\nabla \cdot \boldsymbol\varphi)\di \bb{y} \dd t\\
     -  \int_{Q_T} \mathbb{S}( \nabla\bb{u}): \nabla \boldsymbol\varphi\di \bb{y} \dd t
     +\delta\int_{Q_T}|\bb{u}|^2\bb{u}\cdot\boldsymbol\varphi\di \bb{y} \dd t+\int_{\Gamma_T} \eta_t \psi_t\di x \dd t - \int_{\Gamma_T}\eta_{xx} \psi_{xx}\di x \dd t- \bint_{\Gamma_T} \eta_{tx} \psi_x\di x \dd t\\
     =  -\int_{\Gamma_T} f\psi\di x \dd t - \int_{Q_T} \rho \bb{F}_\delta\cdot \boldsymbol\varphi\di \bb{y} \dd t \label{momeqweak:delta}
 \end{multline}
for all $\boldsymbol\varphi \in C_{\#}^\infty(Q_T)$ and $\psi\in C^\infty_{\#,0}(\Gamma_T)$ such that $\boldsymbol\varphi (t,x,\hat{\eta}(t,x))=\psi(t,x)\bb{e}_2$ on $\Gamma_T$.

\subsection{Limit in the energy inequality}
Our aim here is to pass to the limit in \eqref{EEtestfction_lim1}, where $\phi \in C^\infty_\#(0,T)$, $\phi \geq 0$. First, it is easy to pass to the limit on the right hand side, in particular the last two terms converge to zero. On the left hand side we simply discard the penalization term 
\begin{equation*}
\frac1{\varepsilon}\int_{\Gamma_T} \phi|\bb{v}_\ep - (\eta_\ep)_t \bb{e}_2|^2\di x \dd t,    
\end{equation*}
because it is obviously non-negative. We apply the same argument for the terms
\begin{equation*}
    \varepsilon\gamma\int_{Q_T} \phi\rho_\ep^{\gamma-2}|\nabla \rho_\ep|^2\di \bb{y} \dd t + \varepsilon\delta a\int_{Q_T} \phi\rho_\ep^{a-2}|\nabla \rho_\ep|^2\di \bb{y} \dd t.
\end{equation*}
The uniform bounds \eqref{uniform:est2} and \eqref{Edeltadef} imply that
\begin{equation*}
    \frac{\varepsilon\gamma}{\gamma-1} \int_{Q_T}\phi\rho_\ep^{\gamma}\di \bb{y} \dd t + \frac{\varepsilon\delta a}{a-1} \int_{Q_T}\phi\rho_\ep^{a} \to 0\di \bb{y} \dd t.
\end{equation*}
Next, we use the weak lower semicontinuity of convex functions to pass to the limit in the terms
\begin{equation*}
    \int_{Q_T} \phi\mathbb{S}(\nabla \bb{u}_\ep):\nabla \bb{u}_\ep \di \bb{y} \dd t +\delta \int_{Q_T}\phi|\bb{u}_\ep|^4\di \bb{y} \dd t+\int_{\Gamma_T}\phi|(\eta_\ep)_{tx} |^2 \di x \dd t.
\end{equation*}
It remains to identify the limit of the first term in \eqref{EEtestfction_lim1}, namely
\begin{multline*}
    \int_0^T\phi_t(t) E_\delta(t)\di t  \\
    = \int_{Q_T}\left( \frac12\rho_\ep|\bb{u}_\ep|^2+\frac{\delta}{2}|\bb{u}_\ep|^2+ \frac{1}{\gamma-1}\rho_\ep^\gamma+\frac{\delta}{a-1}\rho_\ep^a\right)\phi_t\di \bb{y} \dd t+\int_{\Gamma_T}\left( \frac12|(\eta_\ep)_{t}|^2+\frac12|(\eta_\ep)_{xx}|^2\right)\phi_t\di x \dd t.
\end{multline*}
We use the same arguments as when passing to the limit in the convective term in the coupled momentum equation to obtain
\begin{equation}
    \frac 12\int_{Q_T}(\delta|\bb{u}_\varepsilon|^2+ \rho_\ep|\bb{u}_\varepsilon|^2)\phi_t(t)\di \bb{y} \dd t\to \frac 12\int_{Q_T}(\delta|\bb{u}|^2+\rho|\bb{u}|^2)\phi_t(t)\di \bb{y} \dd t. \label{kin:en:conv:delta}
\end{equation}
Moreover, the a.e. convergence of $\{\rho_\varepsilon\}_{\varepsilon>0}$ and equiintegrability of $\{\rho_\varepsilon^a\}_{\varepsilon>0}$, imply
\begin{equation*}
\int_{Q_T}\left(\frac{1}{\gamma-1}\rho_\varepsilon^{\gamma}+\frac{\delta}{a-1}\rho_\varepsilon^a\right)\phi_t(t)\di \bb{y} \dd t\to \int_{Q_T}\left(\frac{1}{\gamma-1}\rho^{\gamma}+\frac{\delta}{a-1}\rho^a\right)\phi_t(t)\di \bb{y} \dd t.    
\end{equation*}
The bound on $\partial_{tx}\eta_{\varepsilon}$ in $L^2(\Gamma_T)$ and \eqref{improved:eta} imply $\partial_{xx} \eta_\varepsilon \to \partial_{xx} \eta$ strongly in $L^2(\Gamma_T)$ so 
\begin{equation*}
\frac 12 \int_{\Gamma_T} |\partial_{xx}\eta_\varepsilon|^2\phi_t(t)\di x \dd t \to \frac 12\int_{\Gamma_T}|\partial_{xx}\eta|^2\phi_t(t)\di x \dd t.    
\end{equation*}
It only remains to prove the convergence of the term involving the square of the time derivative of $\eta_\ep$. First, we choose $(\boldsymbol\varphi,\psi) = (\eta_\varepsilon\bb{e}_2,\eta_\varepsilon)$ in \eqref{coupledback:mom} and $(\boldsymbol\varphi,\psi) = (\eta\bb{e}_2,\eta)$ in \eqref{momeqweak:delta} and we compare the two identities to conclude that
\begin{equation}
    \int_{Q_T}(\delta+\rho_\varepsilon)\bb{u}_\varepsilon \cdot \partial_t \eta_\varepsilon \bb{e}_2\di \bb{y} \dd t + \int_{\Gamma_T} |\partial_t \eta_\varepsilon|^2\di x \dd t \to \int_{Q_T}(\delta+\rho)\bb{u} \cdot \partial_t \eta \bb{e}_2\di \bb{y} \dd t + \int_{\Gamma_T} |\partial_t \eta|^2\di x \dd t. \label{etat:strong:delta1}
\end{equation}
Moreover, the strong convergence of $(\delta+\rho_\varepsilon)\bb{u}_\varepsilon\to (\delta+\rho)\bb{u}$ in $L^2(0,T; H^{-\frac12}(\Omega^{\eta}(t)))$ and the weak convergence of $\bb{u}_\varepsilon - \text{Ext}[\bb{v}_\varepsilon]$ to $\bb{u} - \eta_t \bb{e}_2$ in $L^2(0,T; H_0^{\frac12}(\Omega^\eta(t)))$ where $\text{Ext}[\bb{v}_\varepsilon](t,x,z) = \bb{v}_\varepsilon(t,x)$ imply
\begin{multline}
   \int_{Q_T}(\delta+\rho_\varepsilon)\bb{u}_\varepsilon \cdot \big(\bb{u}_\varepsilon - \partial_t \eta_\varepsilon\bb{e}_2\big)\di \bb{y} \dd t \\
   = \int_{Q_T}(\delta+\rho_\varepsilon)\bb{u}_\varepsilon \cdot (\bb{u}_\varepsilon - \text{Ext}[\bb{v}_\varepsilon])\di \bb{y} \dd t + \int_{Q_T}(\delta+\rho_\varepsilon)\bb{u}_\varepsilon \cdot \underbrace{(\text{Ext}[\bb{v}_\varepsilon]-\partial_t \eta_\varepsilon\bb{e}_2)}_{\to 0} \di \bb{y} \dd t \\
   \to \int_{Q_T}(\delta+\rho)\bb{u}\cdot (\bb{u} - \eta_t\bb{e}_2)\di \bb{y} \dd t. \label{etat:strong:delta2}
\end{multline}
We sum up \eqref{etat:strong:delta1} and \eqref{etat:strong:delta2} and by \eqref{kin:en:conv:delta}  we deduce
\begin{equation*}
    \frac 12 \int_{\Gamma_T}|\partial_t \eta_\varepsilon|^2\phi_t(t)\di x \dd t \to \frac 12 \int_{\Gamma_T}|\partial_t \eta|^2\phi_t(t)\di x \dd t. 
\end{equation*}
Thus,  $(\rho,\bb{u},\eta)$ satisfies
\begin{multline}
     - \int_0^T\phi_t(t) E_\delta(t)\di t +\int_{Q_T} \phi\mathbb{S}(\nabla \bb{u}):\nabla \bb{u}\di \bb{y} \dd t+\delta \int_{Q_T}\phi|\bb{u}|^4\di \bb{y} \dd t+\int_{\Gamma_T}\phi|\eta_{tx} |^2\di x \dd t  \\
      \qquad \leq \int_{\Gamma_T} \phi f\eta_t\di x \dd t  + \int_{Q_T}\phi \rho\bb{u}\cdot\bb{F}_\delta \di \bb{y} \dd t\label{EEtestfction_lim2}
\end{multline}
for all $\phi \in C^\infty_\#(0,T)$, $\phi \geq 0$.

\subsection{Estimates independent of $\delta$}\label{sec:deltaest}
At this point, one can adjust the calculations from Section \ref{aprioriest} to take into account terms with $\delta$ in \eqref{momeqweak:delta} in order to deduce estimates independent of $\delta$. We only list main changes with respect to Section \ref{aprioriest} here. The starting point is the energy inequality \eqref{EEtestfction_lim2}, where we first use test function $\phi = 1$ and follow Section \ref{partI} to get
\begin{equation}\label{diss:est2}
   \delta\|\bb{u}\|_{L^4(Q_T)}^4  + \| \bb{u}\|_{L^2(0,T;H^1(\Omega))}^2+
    \|\eta_t \|_{L^2(0,T;H^1(\Gamma))}^2 \leq C(\kappa)(1+\mathcal{E}_\delta^\kappa).
\end{equation}
Next, using the notation for $E_\delta(t)$ and $\mathcal{E}_\delta$ introduced in \eqref{Edelta.def} and \eqref{Edeltadef} respectively, we take a sequence of test functions $\phi_k \to \chi_{[s,t]}$, pass to the limit with $k \to \infty$ and using calculations of Section \ref{partII} we get \begin{equation*}
  \mathcal{E}_\delta \leq C_0\left(1+\int_0^T E_\delta(s)\di s\right).
\end{equation*}
 All terms are handled similarly to their counterparts in  Section \ref{eta:est:sec}, there are however two additional terms with respect to \eqref{eq:est_etaxx1}. These are treated as follows
\begin{multline*}
  \delta\left|\int_{Q_T} |\bb{u}|^2\bb{u}\cdot{\eta\bb{e}_2}\di \bb{y} \dd t\right| \leq \delta\|\bb{u}\|_{L^4(Q_T)}^3\|\eta\|_{L^4(\Gamma_T)} \\
   \leq C(\kappa)(1+\mathcal{E}_\delta^{\frac{3\kappa}{4}})(\|\eta_t\|_{L^2(0,T;L^4(\Gamma))} + \|\eta\|_{L^2(0,T;L^4(\Gamma))}) \leq C(\kappa)(1+\mathcal{E}_\delta^{\frac{3\kappa}{2}}) + \frac{1}{8}\|\eta_{xx}\|_{L^2(\Gamma_T)}^2,
\end{multline*}
and
\begin{equation*}
    \delta\left|\int_{Q_T} \bb{u} \cdot \eta_t  \bb{e}_2\di \bb{y} \dd t\right| \leq C\|\bb{u}\|_{L^2(0,T;L^q(\Omega^\eta(t)))} \|  \eta_t\|_{L^2(0,T;L^\infty(\Gamma))}\leq C(\kappa)\left(1+\mathcal{E}_\delta^{\kappa}\right).
\end{equation*}
Eventually we recover
\begin{equation*}
    \int_{\Gamma_T}| \eta_{xx}  |^2\di x \dd t \leq  C(\kappa)(1+\mathcal{E}_\delta^{3\kappa}).
\end{equation*}
Finally, 
 \eqref{eq:mainBog} contain the additional term $\delta\int_{Q_T^\eta}\rho^{a+\alpha} \di \bb{y} \dd t$ on the left hand side and four more terms on the right hand side. Two terms arise from the $\delta\rho^a$ in the pressure and these terms are estimated exactly as in \eqref{eq:Bogest11} and \eqref{eq:Bogest12}. Next, similarly as in \eqref{eq:Bogest13}
\begin{equation*}
    \delta\left|\int_{Q_T^\eta} \bb{u}\cdot \partial_t \boldsymbol\varphi_h\di \bb{y} \dd t\right| \leq C(\kappa)\left(1+ \mathcal{E}_\delta^{\frac{\alpha}{\gamma}+\kappa}\right),
\end{equation*}
and 
\begin{equation*}
    \delta\left|\int_{Q_T^\eta} |\bb{u}|^2\bb{u}\cdot \boldsymbol\varphi_h\di \bb{y} \dd t\right| \leq C(1+  \mathcal{E}_\delta^{\frac34})
\end{equation*}
We then continue as in Section \ref{Bog:section} and end with \eqref{improved:int:est} and thanks to the choice of parameters $\alpha, \kappa$ we get \eqref{interest}. We want a similar bound also for $\delta\rho^a$, however we can not use the same combination of parameters $\alpha$ and $\kappa$, because the inequality \eqref{eps:epsprime} might not hold if $\gamma$ is replaced by $a$. Therefore, we next set $\bar{\kappa} := \frac{1}{5(a-1)}$ and $\bar{\alpha} := \frac 25$, repeat the calculations of Sections \ref{partI}-\ref{eta:est:sec} and Section \ref{Bog:section} in order to deduce
\begin{equation*}
    \delta\int_0^T \int_{\{\eta+h<z<\eta+2H-h\}} \rho^{a+\bar{\alpha}}\di \bb{y} \dd t\leq  \delta\int_{Q_T^\eta} \rho^{a+\bar{\alpha}}\phi_h\di \bb{y} \dd t \leq C(\bar{\kappa})\left(1+\mathcal{E}_\delta^{1+\frac{3\bar{\kappa}}2}\right).
\end{equation*}
By interpolation 
\begin{equation}
    \delta\int_0^T \int_{\{\eta+h<z<\eta+2H-h\}} \rho^{a}\di \bb{y} \dd t \leq C(\bar{\kappa})\left(1+\mathcal{E_\delta}^{1-\bar{\kappa}'}\right)\label{interest2}, 
\end{equation}
where
\begin{equation*}
    \bar{\kappa}':= 1 - \left(1+\frac{3\bar{\kappa}}2\right)\frac{a-1}{a+\bar{\alpha}-1}.
\end{equation*}
We continue with estimates of the pressure near the boundary using the function \eqref{eq:varphi_h_def2}. Again, we encounter some additional terms in equation \eqref{eq:Bogbdry}. To be more precise, terms $\delta\rho^a$ appear both on the left hand side and in the first term on the right hand side. The left hand side provides the information we seek, while the term on the right hand side is bounded using \eqref{interest2}. The integrals of $\delta\bb{u}\cdot\partial_t(\varphi_h\bb{e}_2)$ and $\delta|\bb{u}|^2\bb{u}\cdot(\varphi_h\bb{e}_2)$ yield the powers $\mathcal{E}_\delta^\kappa$ and $\mathcal{E}_\delta^{\frac 34}$, respectively. Hence, we conclude that there exists $\kappa'' > 0$ such that 
\begin{equation*}
    \int_{Q_T} \rho^{\gamma} + \delta\rho^a \di \bb{y} \dd t\leq C\left(1+\mathcal{E}^{1-\kappa''}\right).
\end{equation*}
Finally, in Section \ref{35sec} we estimate $\delta\int_{Q_T} |\bb{u}|^2$ by \eqref{diss:est2} and we obtain 
\begin{equation}\label{eq:finalbounds}
    \mathcal{E}_\delta\leq C, \qquad \int_{Q_T} \mathbb{S}(\nabla \bb{u}):\nabla \bb{u}\di \bb{y} \dd t+\delta \int_{Q_T}|\bb{u}|^4\di \bb{y} \dd t+\int_{\Gamma_T}|\eta_{tx} |^2 \di x \dd t\leq C.  
\end{equation}
Similarly to Section \ref{sec:improvedeta}, we obtain
\begin{equation}\label{eq:improveddelta}
    \|\eta\|_{L^2(0,T; H^{2+s}(\Gamma))}^2\leq C,
\end{equation}
for some $s>0$.

\section{Limit $\delta\to 0$}
Denote the solution obtained in previous section as $(\rho_\delta,\bb{u}_\delta,\eta_\delta)$. The goal is to pass to the limit $\delta\to 0$ to conclude that the limiting functions $(\rho,\bb{u},\eta)$ represent a weak solution in the sense of Definition \ref{weak:sol:def}. The uniform estimates deduced in Section \ref{sec:deltaest} give rise to the following convergencies
\begin{equation*}
\begin{split}
    \rho_\delta &\weak \rho \quad \text{ weakly}^* \text{ in } L^\infty_\#(0,T;L^\gamma_\#(\Omega)), \\
    \bb{u}_\delta &\weak \bb{u} \quad\text{ weakly in } L^2_\#(0,T;H^1_\#(\Omega)), \\
    \eta_\delta &\weak \eta \quad \text{ weakly}^* \text{ in } L^\infty_\#(0,T;H^2_{\#}(\Gamma)) \quad \text{ and weakly in } H^1_\#(0,T;H^1_{\#,0}(\Gamma)).
\end{split}
\end{equation*}

\subsection{Limit in the continuity equation}

We employ standard arguments from the existence theory of weak solutions to the compressible Navier-Stokes equations (see i.e. \cite{FeNobook}) to deduce that functions $\rho$ and $\bb{u}$ satisfy the continuity equation in the weak sense, i.e.
\begin{equation*}
     \int_{Q_T} \rho (\partial_t \varphi +\bb{u}\cdot \nabla \varphi)\di x \dd t=0
\end{equation*}
for all $\varphi \in C^{\infty}_\#(Q_T)$. The validity of the renormalized continuity equation remains open at this moment since $\rho$ may not possess enough regularity to use a direct argument.

\subsection{Limit in the coupled momentum equation}

First, the kinematic coupling $\bb{u}(t,x,\hat{\eta}(t,x)) = \eta_t(t,x)\bb{e}_2$ is recovered using Lemma \ref{tracelemma}. Our aim is to pass with $\delta$ to zero in \eqref{momeqweak:delta}. Once again, the terms integrated over $\Gamma_T$ are linear and therefore their limits are straightforward. Estimates \eqref{eq:finalbounds} are enough to identify $0$ as a limit of terms $\int_{Q_T} \delta\bb{u}\cdot\partial_t\boldsymbol\varphi \di \bb{y} \dd t$ and $\int_{Q_T} \delta|\bb{u}|^2\bb{u}\cdot\boldsymbol\varphi\di \bb{y} \dd t$. The limit in the last term on the right hand side is easy. In the remaining terms we follow the existence theory of weak solutions to the compressible Navier-Stokes equations and the main task is to deduce the limit in the pressure term, which is closely related to the validity of the renormalized continuity equation. Both issues are solved by means of the effective viscous flux identity and boundedness of the oscillations defect measure. We get the pointwise convergence of $\rho_\delta \to \rho$ a.e. in $Q_T$ and thus recover both \eqref{reconteqweak} and \eqref{momeqweak}.

\subsection{Limit in the energy inequality}

Finally we need to pass to the limit in \eqref{EEtestfction_lim2} in order to prove \eqref{eq:EE1}. The limits of the terms on the right hand side are simple. On the left hand side we simply discard the term $\delta\int_{Q_T} \phi |\bb{u}|^4$ since it is surely nonnegative and for the second and fourth term on the left hand since we use lower semicontinuity of convex functions. Therefore it remains to deal with the first term on the left hand side. First, the kinetic energy term is treated the same way as the convective term in the coupled momentum equation. Next, it is easy to use \eqref{eq:finalbounds} to pass to zero in the term containing $\delta|\bb{u}|^2$. Pointwise convergence of densities allows us to pass to the limit in the pressure terms of $E_\delta$. Improved estimate \eqref{eq:improveddelta} allows us to pass to the limit in the last term of $E_\delta$, while a similar procedure as in \eqref{etat:strong:delta1}-\eqref{etat:strong:delta2} provides necessary information to pass to the limit in the term $|\eta_t|^2$ of $E_\delta$. Thus we recover \eqref{eq:EE1}.
The validity of \eqref{eq:EEbound} follows from the calculations in Section \ref{aprioriest} with the starting point being the energy inequality \eqref{eq:EE1}.

\vspace{.1in}
\noindent{\bf Acknowledgments:} The work of O. K., V. M. and \v S. N. was supported by Praemium Academiae of \v S. Ne\v casov\' a and by the Czech Science Foundation (GA\v CR) through project GA22-01591S. The Institute of Mathematics, Czech Academy of Sciences, is supported by RVO:67985840.

\Addresses

\end{document}